\documentclass[a4paper, 10pt, twoside, notitlepage]{amsart}

\usepackage{amsmath,amscd}
\usepackage{amssymb}
\usepackage{amsthm}
\usepackage{comment}
\usepackage{graphicx, xcolor}

\usepackage{mathrsfs}
\usepackage[ocgcolorlinks, linkcolor=blue]{hyperref}

\usepackage[ocgcolorlinks,linkcolor=blue]{hyperref}

\newcommand{\LC}{\left(}
\newcommand{\RC}{\right)}

\theoremstyle{plain}
\newtheorem{thm}{Theorem}[section]
\newtheorem{prop}{Proposition}[section]
\newtheorem{lem}[prop]{Lemma}
\newtheorem{cor}[prop]{Corollary}

\newtheorem{defi}[prop]{Definition}
\newtheorem{rmk}[prop]{Remark}

\newtheorem*{thm*}{Theorem}

\numberwithin{equation}{section}
\newcommand {\R} {\mathbb{R}} \newcommand {\Z} {\mathbb{Z}}
 \newcommand {\N} {\mathbb{N}}
 
\newcommand {\p} {\partial}

\newcommand{\eps}{\epsilon}

\newcommand {\supp} {\text{supp}}

\newcommand{\ol}{\overline}

\newcommand{\tbl}{\textcolor{blue}}

\newcommand{\wt}{\widetilde}

\newcommand{\mR}{\mathbb{R}}                    % Formatting for R
\newcommand{\abs}[1]{\lvert #1 \rvert}          % Formatting for the absolute value
\newcommand{\norm}[1]{\left\Vert #1 \right\Vert} 
\newcommand{\re}{\mathrm{Re}}
\newcommand{\im}{\mathrm{Im}}

\newcommand{\ccdot}{\,\cdot\,}
\newcommand{\s}{\hspace{0.5pt}}
\newcommand{\f}[1]{\footnote{\textcolor{blue}{#1}}}
\definecolor{armygreen}{rgb}{0.29, 0.33, 0.13}
\definecolor{ao(english)}{rgb}{0.0, 0.5, 0.0}

\DeclareMathOperator{\Id} {Id}

\title[Inverse problems for fractional power elliptic equations]{Inverse problems for elliptic equations with fractional power type nonlinearities}

\author[Liimatainen]{Tony Liimatainen}
\address{Department of Mathematics and Statistics, University of Jyv\"askyl\"a, Jyv\"askyl\"a, Finland 
\newline\noindent
\noindent Department of Mathematics and Statistics, University of Helsinki, Helsinki, Finland}
\curraddr{}
\email{tony.liimatainen@helsinki.fi}

\author[Lin]{Yi-Hsuan Lin}
\address{Department of Applied Mathematics, National Chiao Tung University, Hsinchu, Taiwan}
\curraddr{}
\email{yihsuanlin3@gmail.com}

\author[Salo]{Mikko Salo}
\address{Department of Mathematics and Statistics, University of Jyv\"askyl\"a, Jyv\"askyl\"a, Finland}
\curraddr{}
\email{mikko.j.salo@jyu.fi}

\author[Tyni]{Teemu Tyni}
\address{Department of Mathematics and Statistics, University of Helsinki, Helsinki, Finland}
\curraddr{}
\email{teemu.tyni@helsinki.fi}

\begin{document}

\maketitle
\begin{abstract}
	
	We study inverse problems for semilinear elliptic equations with fractional power type nonlinearities. Our arguments are based on the higher order linearization method, which helps us to solve inverse problems for certain nonlinear equations in cases where the solution for a corresponding linear equation is not known. By using a fractional order adaptation of this method, we show that the results of \cite{LLLS2019nonlinear, LLLS2019partial} remain valid for general power type nonlinearities. %In addition, we also study Calder\'on's type problem and simultaneously recovering inverse problem with partial data, which still remain open to their linear counterparts.

		\medskip
	
	\noindent{\bf Keywords.} Inverse boundary value problem, Calder\'on problem, partial data, semilinear elliptic equations, higher order linearization, transversally anisotropic manifold.
	
	%\noindent{\bf Mathematics Subject Classification (2010)}: 
\end{abstract}

\setcounter{tocdepth}{1}

\tableofcontents

\section{Introduction}\label{Sec 1}
In this work we study inverse problems for semilinear elliptic equations with fractional power type nonlinearities, extending the earlier results in \cite{LLLS2019nonlinear, LLLS2019partial} from integer powers to fractional powers. Here, when we say $r$ is fractional we mean $r\in \R\setminus\Z$. Let $r>1$ be fractional and let $\Omega \subset \R^n$ be a bounded domain with $C^\infty$-smooth boundary $\p \Omega$, for $n \geq 2$.
Consider the semilinear elliptic equation 
\begin{align}\label{Main equation}
	\begin{cases}
	\Delta u + q(x)|u|^{r-1}u=0 & \text{ in }\Omega, \\
	u=f & \text{ on }\p \Omega,
	\end{cases}
\end{align}
where $q\in C^\alpha(\overline{\Omega})$ is a potential function and $C^\alpha$ is the space of $\alpha$-H\"older continuous functions. By assuming a suitable \emph{smallness condition} on the boundary data $f$, one can obtain the well-posedness of the Dirichlet problem \eqref{Main equation} for small solutions (see Section \ref{Sec 2}). One can then define the corresponding \emph{Dirichlet-to-Neumann} (DN) map $\Lambda_q$ of \eqref{Main equation} by 
\begin{align*}
	\Lambda_q: C^{2,\alpha}(\p \Omega)\to C^{1,\alpha}(\p \Omega),\qquad  f \mapsto \left. \p_\nu u _f \right|_{\p \Omega},
\end{align*}
for some $0<\alpha <1$, where $u_f \in C^{2,\alpha}(\overline{\Omega})$ is the unique small solution of \eqref{Main equation}, and $\nu$ is the unit outer normal on $\p \Omega$. We will consider the following problem:

\vspace{2mm}

\noindent $\bullet$ \textbf{Inverse Problem 1:} Determine the potential $q$ from the knowledge of $\Lambda_q$.

\vspace{2mm}

A typical method in the study of inverse boundary value problems for nonlinear elliptic equations was initiated by Isakov \cite{isakov1993uniqueness_parabolic}, where he introduced the first linearization of the given (nonlinear) DN map. More precisely, the first linearization allows one to reduce the nonlinear equations to the linear equations, and one can adapt some known results for the linear equations to solve certain inverse problems for the nonlinear equations. Meanwhile, the second order linearization has been successfully applied in solving inverse problems, see \cite{AYT2017direct,CNV2019reconstruction,KN002,sun1996quasilinear,sun1997inverse}.

Throughout this paper the number $r>1$ is fractional, and the solution $u$ is real valued but may change sign, so it is natural to consider $q(x)|u|^{r-1}u$ instead of $q(x) u^r$ to have well-defined nonlinear term. 
Note also that at least when $n=1$ the case $0<r<1$ would roughly correspond to the second order differential equation $u''=F(u)$, where $F$ is not Lipschitz. In this case, it is well-known that uniqueness of solutions can fail, so the assumption $r>1$ is reasonable. Let us write $r=k+\alpha>1$ for some $k\in \N$ and $\alpha\in (0,1)$ in the rest of this work.

In case of $r=m\in \N$ and nonlinear term $q(x)u^m$, corresponding inverse problems were first investigated in \cite{FO20,LLLS2019nonlinear}, and related problems have been further studied in many works. For example, the articles \cite{LLLS2019partial,KU2019remark,KU2019partial} studied related inverse problems for semilinear elliptic equations with partial data. In \cite{LL2020inverse,lin2020monotonicity,LO2020fractional_lower}, the authors studied inverse problems for fractional semilinear elliptic equations. In \cite{LZ2020partial,krupchyk2020inverse,carstea2020inverse,kian2020partial}, the authors studied partial data inverse problems for the nonlinear magnetic Schr\"odinger and conductivity equations. The nonlinearities in these articles are typically integer power type, or holomorphic in $u$ and $\nabla u$ (i.e.\ sums of integer powers).

The main tool in solving these inverse problems is based on the \emph{higher order linearization} technique, where one introduces extra small parameters for the Dirichlet data to reduce inverse problems for nonlinear elliptic equations into statements involving solutions of simpler linear elliptic equations. In the case of nonlinearity $q(x) u^m$ where $m \in \N$, this just means that we are looking at the $m$th order Fr\'echet derivative of the nonlinear measurement operator. 
For a nonlinearity of fractional order $r = k+\alpha$, we will in some sense need to use the $\alpha$th fractional derivative of the $k$th Fr\'echet derivative instead. A somewhat related method was used in \cite{carstea2020recovery} for a $p$-Laplace type equation. Thanks to the higher order linearization method, one may solve related inverse problems for certain semilinear elliptic equations in cases where the analogous problems for the corresponding linear equations still remain open.

Let us state our first main result to answer Inverse Problem 1:

\begin{thm}[The Calder\'on problem with full data]\label{Main Thm 1}
Let $\Omega \subset \R^n$ be a connected bounded domain with $C^\infty$-smooth boundary $\p \Omega$, for $n \geq 2$. Let $r>1$ be a fractional number, $q_j\in C^\alpha (\overline{\Omega})$ for some $0<\alpha<1$, and $\Lambda_{q_j}$ be the DN map of 
\begin{align*}
\begin{cases}
	\Delta u_j +q_j |u_j|^{r-1}u_j =0 & \text{ in }\Omega, \\
	u_j =f & \text{ on }\p \Omega,
\end{cases}
\end{align*}
for $j=1,2$. Assume that $	\Lambda_{q_1}(f)=\Lambda_{q_2}(f)$, for all $f\in C^{2,\alpha}(\p \Omega)$ with $\norm{f}_{C^{2,\alpha}(\p \Omega)}<\delta$, where $\delta >0$ is a sufficiently small number. Then 
\begin{align*}
 q_1 =q_2 \text{ in }\Omega.
\end{align*}
Moreover, in dimensions $n\geq 3$ the statement holds true if we only assume that $\Lambda_{q_1}(f)=\Lambda_{q_2}(f)$ whenever $\norm{f}_{C^{2,\alpha}(\p \Omega)}<\delta$ and $f \geq 0$.
%under the weaker assumption that $f>0$.
\end{thm}

We remark that in certain applications it is natural to consider nonnegative Dirichlet data (see e.g.\ \cite{RenZhang}). Theorem \ref{Main Thm 1} applies in this case when $n \geq 3$. However, the methods for proving the other main theorems in this paper require sign-changing solutions, and we do not know if those results are valid if one only has access to measurements for nonnegative Dirichlet data.

We briefly explain the higher order linearization in the fractional power case. 
Let $(M,g)$ be a compact $C^\infty$ Riemannian manifold with a $C^\infty$ smooth boundary $\p M$. Recall that $\Delta_g$ is the Laplace-Beltrami operator, given in local coordinates by 
\[
\Delta_g u=\frac{1}{\det(g)^{1/2}}\sum_{a,b=1}^n\frac{\p}{\p x_a}\left( \det(g)^{1/2} g^{ab}\frac{\p u}{\p x_b}\right),
\]
where $g=(g_{ab}(x))$ and $g^{-1}=(g^{ab}(x))$. Throughout this work, we assume that $g=(g_{ab})$ is uniformly elliptic.
Let $q \in C^{\alpha}(M)$. In Proposition~\ref{Prop: derivs_and_integral_formula} we will see that by setting the Dirichlet data as 
$$
f=  \eps_0f_0+\ldots+\eps_kf_k 
$$ 
and differentiating the equation \eqref{Main equation} with respect to $\eps'=(\eps_1,\ldots,\eps_k)$ we obtain a new equation
	\begin{equation} \label{laplace_um_equation}
		\Delta_g w^{\eps_0}(x) = - \p_{\eps_1} \cdots \left. \p_{\eps_k} \LC  q(x)\abs{u_f}^{r-1}u_f\RC \right|_{\eps'=0}  \text{ in }M,
	\end{equation}
where $w^{\eps_0}:=\left. \p_{\eps_1} \cdots \p_{\eps_k}u_f\right|_{\eps'=0}$ and $\left. w^{\eps_0}\right|_{\p M}=\left. \eps_0f_0\right|_{\p M}$.

Furthermore, eliminating $\eps_0^{\alpha}$ on the both sides of \eqref{laplace_um_equation}, by taking the limit $\eps_0\to 0$, we get
	\[
	\eps_0^{-\alpha}w^{\eps_0}\to w \text{ in }C^{2,\alpha}(M), \quad \text{ as }\eps_0\to 0,
	\]
	where $w$ solves
	%~\f{Use $h_0=\eps_0f_0$ instead.}
	\begin{equation*}
		\Delta_g w= c_r q(x) \mathrm{sgn}(v_0)^{k-1} \abs{v_{0}}^\alpha v_{1} \cdots v_{k} \text{ in $M$}.
	\end{equation*}
Here $c_r$ is the constant given by $c_r=-r(r-1)\cdots (r-(k-1))$, $\mathrm{sgn}(v_0(x))$ is the sign of $v_0(x)$, and the functions $v_{\ell}$ are harmonic in $M$ with the corresponding boundary values $f_\ell$, for $\ell=0,1,\ldots, k$. Moreover, we will multiply this equation by an extra auxiliary harmonic function $v_{k+1}$ in $M$ with its boundary data $\left. v_{k+1}\right|_{\p M}=f_{k+1}$. Now integrating over $M$ and using integration by parts, we see that from the knowledge of the DN map for the equation $\Delta_g u + q(x) \abs{u}^{r-1} u = 0$ in $M$ it is possible to determine the integrals 
\[
c_r \int_M q(x) \mathrm{sgn}(v_0)^{k-1} \abs{v_{0}}^\alpha v_{1} \cdots v_{k+1} \,dV.
\]
It thus suffices to choose the boundary data $f_\ell$ for $\ell=0,1,\ldots , k$, so that $v_0 \neq 0$ in $M$ and the scalar products $v_{1}\cdots v_{k+1}$ become dense in a suitable function space. This recovers the function $q$ (see Sections \ref{Sec 3} and \ref{Sec 4}).

Next we study the Calder\'on problem with partial data for elliptic equations with fractional power type nonlinearities. 
Let $\Omega \subset \R^n$ be a connected bounded domain, and $\Gamma \subset \p \Omega$ be a nonempty relatively open subset.
By using the well-posedness of \eqref{Main equation} (Proposition \ref{Prop: wellposedness_and_expansion}), one can define the corresponding partial DN map $\Lambda_q^\Gamma$ of \eqref{Main equation} by 
\begin{align*}
\Lambda_q^\Gamma: C^{2,\alpha}_0(\Gamma)\to C^{1,\alpha}(\Gamma),\qquad  f \mapsto \left. \p_\nu u _f \right|_{\Gamma},
\end{align*}
for some $0<\alpha <1$, where $u_f \in C^{2,\alpha}(\overline{\Omega})$ is the unique (small) solution of \eqref{Main equation} (see Section \ref{Sec 2}) with $f\in C^{2,\alpha}_0(\Gamma) $. Then our second question is:

   \vspace{2mm}
   
   \noindent $\bullet$ \textbf{Inverse Problem 2:} Determine the potential $q$ from the knowledge of $\Lambda_q^\Gamma$.
   
   \vspace{2mm}

Our second main result is to solve Inverse Problem 2:

\begin{thm}[Partial data]\label{Main Thm 2}
	Let $\Omega \subset \R^n$ be a connected bounded domain with $C^\infty$-smooth boundary $\p \Omega$, for $n \geq 2$, and $\Gamma \subset \p \Omega$ be a nonempty relatively open subset. Let $r>1$ be a fractional number, $q_j\in C^\alpha (\overline{\Omega})$ for some $0<\alpha<1$, and $\Lambda_{q_j}^\Gamma$ be the DN map of 
	\begin{align*}
	\begin{cases}
	\Delta u_j +q_j |u_j|^{r-1}u_j =0 & \text{ in }\Omega, \\
	u_j =f & \text{ on }\p \Omega,
	\end{cases}
	\end{align*}
	for $j=1,2$. If $	\Lambda_{q_1}^\Gamma(f)=\Lambda_{q_2}^\Gamma(f)$, for all $f\in C^{2,\alpha}_0(\Gamma)$ with $\norm{f}_{C^{2,\alpha}_0(\Gamma)}<\delta$, where $\delta >0$ is a sufficiently small number, then
	\begin{align*}
 q_1 =q_2 \text{ in }\Omega.
	\end{align*}
\end{thm}

Moreover, one can consider more general nonlinear terms that are (asymptotic) sums of homogeneous functions. Let $\Omega \subset \R^n$ be a bounded domain with $C^\infty$-smooth boundary $\p \Omega$.

\begin{defi}\label{polyhomogeneous}
Let $r_l$, $l \geq 1$, be real numbers with $1 < r_1 < r_2 < \ldots$, and let $0 < \alpha < 1$. A function $a=a(x,y): \overline{\Omega}\times \R \to \R$ is polyhomogeneous, written 
 \[
 a(x,y) \sim \sum_{l=1}^\infty b_l (x,y),
 \]
 if each $b_l(\,\cdot\,,y) \in C^{\alpha}(\ol{\Omega})$ is positively homogeneous of degree $r_l$ with respect to the $y$-variable, and if for any $N \geq 1$ there is $C_N > 0$ so that the function $\beta_N := a - \displaystyle\sum_{l=1}^{N-1} b_l$ (with $\beta_1 = a$) is in $C^{1,\alpha}_{\mathrm{loc}}(\mR, C^{\alpha}(\ol{\Omega}))$ and satisfies 
 \begin{equation}\label{beta_norm_bound}
 \norm{\beta_N(\,\cdot\,,y)}_{C^{\alpha}(\ol{\Omega})} + \abs{y} \norm{\p_y \beta_N(\,\cdot\,,y)}_{C^{\alpha}(\ol{\Omega})} \leq C_N \abs{y}^{r_N}, \qquad \abs{y} \leq 1.
 \end{equation}
We will assume that $1+\alpha \leq r_1$ (this can be arranged by decreasing $\alpha$).
%{\color{blue} We will write $r_l = k_l + \alpha_l$ where $k_l \in\N$ and $0<\alpha_l\leq 1$. We also assume that $\alpha_1 \geq \alpha > 0$.}
\end{defi}

%In the rest of the introduction, the coefficient $a=a(x,y)$ is always assumed to be polyhomogeneous as above, with the preceding conditions holding for some numbers $r_l=k_l + \alpha_l$ with $k_l \geq 1$, $\alpha_l\in [0,1)$ for $l\in \N$, and $\alpha=\alpha_1 \in (0,1)$. 
Note that the above definition (using $N=1$) implies that 
\begin{equation} \label{phg_zero}
a(x,0) = \p_y a(x,0) = 0.
\end{equation}
A typical example of polyhomogeneous function $a(x,y)$ is a finite sum 
\[
a(x,y)=\sum_{l=1}^m q_l(x) f_l(y),
\]
where $q_l(x)\in C^{\alpha}(\overline{\Omega})$ and $f_l(y)$ is positively homogeneous of degree $r_l$, i.e. $f_l(\lambda y) = \lambda^{r_l} f_l(y)$ for $y \in \mR$ and $\lambda > 0$. One could also consider infinite sums of this type. In fact, functions $a(x,y)$ that are $C^{\alpha}$ in $x$, holomorphic or antiholomorphic in $y$, and satisfy \eqref{phg_zero} are polyhomogeneous with $r_l = l+1$ just by using Taylor expansions. It is worth  emphasizing that since we are always considering small solutions, only the behaviour for small $\abs{y}$ plays a role.

We also mention that the function $f(y) = |y|^{r-1}y$ , at least roughly speaking, encompasses all positively homogeneous functions. Indeed, if $f$ is positively homogeneous of degree $r>0$, then $f$ is of the form
\begin{align*}
	f(y)=
\begin{cases}
	y^r f(1), &\text{ if }y\geq 0,\\
	f(-|y|)=|y|^r f(-1), & \text{ if }y<0.
\end{cases}
\end{align*}
The case $f(y) = |y|^{r-1}y$ is obtained by taking $f(1)=1$ and $f(-1)=-1$. This computation also shows that if $r = k+\alpha$ where $k \geq 1$ and $\alpha \in (0,1)$, then $f(y)$ is $C^k$ and $f^{(k)}(y)$ is $C^{\alpha}$.

Let us consider the following Dirichlet problem in a bounded smooth domain $ \Omega \subset \R^n$
\begin{align}\label{Dirichlet problem for general coefficients}
	\begin{cases}
		\Delta u + a(x,u)=0 &\text{ in }\Omega, \\
		u=f &\text{ on }\p \Omega,
	\end{cases}
\end{align}
where $a=a(x,y)$ is a polyhomogeneous function given by Definition \ref{polyhomogeneous}.
By Proposition \ref{Prop: wellposedness_and_expansion}, for any sufficiently small Dirichlet data $f\in C^{2,\alpha}_0(\Gamma)$ with $\Gamma \subset \p \Omega$, one can define the corresponding (partial) DN map via 
\begin{align*}
	\Lambda_a^\Gamma: C^{2,\alpha}_0(\Gamma)\to C^{1,\alpha}(\Gamma),\qquad  f \mapsto \left. \p_\nu u _f \right|_{\Gamma},
\end{align*}
for some $0<\alpha <1$, where $u_f \in C^{2,\alpha}(\overline{\Omega})$ is the unique small solution of \eqref{Dirichlet problem for general coefficients}.
The inverse problem is to determine the unknown function $a(x,y)$.

\begin{thm}[Partial data for general coefficients]\label{Main Thm 3}
	%~\f{What are examples of polyhomogeneous functions of $1$-variable? Give nontrivial examples before this theorem.}
Let $\Omega \subset \R^n$ be a connected bounded domain with $C^\infty$-smooth boundary $\p \Omega$, for $n \geq 2$, and $\Gamma \subset \p \Omega$ be a nonempty relatively open subset. Let us consider the equations %et $a_j(x,z)$ be $C^\infty$ functions in $x,z$ satisfying \eqref{condition of a 1} for $j=1,2$. 
	\begin{align}\label{eq:sum_of_bj}
		\Delta u +a_j (x,u)=0 \text{ in }\Omega,
	\end{align}
	for $j=1,2$, where $a_j(x,y) \sim \displaystyle \sum_{l=1}^\infty b_{j,l}(x,y)$ is polyhomogeneous in the sense of Definition~\ref{polyhomogeneous} where the orders $1 < r_1 < r_2 < \ldots$ are the same for $j=1,2$. Let $\Lambda_{a_j}^\Gamma:C^{2,\alpha}_0(\Gamma)\to C^{1,\alpha}(\Gamma)$ be the (partial) DN maps of~\eqref{eq:sum_of_bj}, for $j=1,2$.  Assume that 
	\[
	\Lambda_{a_1}^\Gamma(f)=\Lambda_{a_2}^\Gamma(f) ,
	\]
	for all $f\in C^{2,\alpha}_0(\Gamma)$ with $\norm{f}_{C^{2,\alpha}_0(\Gamma)}<\delta$, where $\delta>0$ is a sufficiently small number. Then we have
	%\f{We probably can replace $k_l,\alpha_k$ below by any $k,\alpha$ such that $k+\alpha\leq k_K+\alpha_K$.}
	\begin{align*}%\label{taylor_series_agree}
		b_{1,l}(x,y)  = b_{2,l}(x,y), \quad \text{ for } x\in \Omega, \ y\in \R \text{ and } l \in \N .%\lim_{\eps_0\to 0}\LC \eps_0^{-\alpha_l}\p^{k_l}_y a_2 (x,0) \RC \text{ in }\Omega 
		%,\quad \text{ for }l\in \N.
	\end{align*}
	In particular, if $b_{j,l}$ is of the form $b_{j,l}(x,y)=q_{j,l}\abs{y}^{r_l-1}y$, where $q_{j,l}(x)\in C^{\alpha}(\overline{\Omega})$, then 
	\[
	 q_{1,l}(x)=q_{2,l}(x) \text{ in }\Omega, \quad  \text{ for } l \in \N.
	\]
\end{thm}

Theorem~\ref{Main Thm 3} corresponds to the recovery of the coefficients of the asymptotic series expansion of $a(x,y)$ in the $y$-variable. Note that numbers $r_1,r_2,\ldots$ could also be integers $\geq 2$. Therefore, we can regard Theorem \ref{Main Thm 3} as a generalization of the corresponding Euclidean results in \cite{LLLS2019nonlinear, LLLS2019partial}.

Inspired by the partial data results of inverse problems for semilinear elliptic equations \cite{LLLS2019partial,KU2019partial}, one can also consider the inverse boundary value problem of recovering an obstacle and coefficients simultaneously. Let $\Omega \subset\R^n$ be a bounded domain with a connected $C^\infty$-smooth boundary $\p \Omega$. Let $D\Subset \Omega$ be an open set with $C^\infty$-smooth boundary $\p D$ such that $\Omega \setminus \overline{D}$ is connected. Consider the boundary value problem 
\begin{align}\label{Main equation with obstacle}
\begin{cases}
\Delta u + a(x,u)=0 &\text{ in }\Omega \setminus \overline{D} ,\\
u=0 &\text{ on }\p D, \\
u=f &\text{ on }\p \Omega,
\end{cases}
\end{align}
where $a=a(x,y)$ is a polyhomogeneous function defined via Definition \ref{polyhomogeneous}, for $x\in \Omega\setminus \overline{D}$.

As shown in Proposition \ref{Prop: wellposedness_and_expansion}, given any Dirichlet data $f\in C^{2,\alpha}(\p \Omega)$ with $\norm{f}_{C^{2,\alpha}(\p \Omega)}<\delta$, for some sufficiently small number $\delta>0$, the equation \eqref{Main equation with obstacle} is well-posed and  admits a unique (small) solution $u\in C^{2,\alpha}(\overline{\Omega}\setminus D)$. Let $\Gamma \subset \p \Omega$ be an arbitrarily nonempty relatively open subset, then we can define the corresponding partial DN map $\Lambda_{a,D}^\Gamma$ by 
\begin{align*}
\Lambda_{a,D}^\Gamma : C^{2,\alpha}(\Gamma)\to C^{1,\alpha}(\Gamma), \qquad f\mapsto \left. \p_\nu u_f\right|_{\Gamma},
\end{align*}
for any $f\in C^{2,\alpha}_0(\Gamma)$ with sufficiently small $\norm{f}_{C^{2,\alpha}_0 (\Gamma)}$, where $u_f\in C^{2,\alpha}(\overline{\Omega}\setminus D)$ is the unique solution of \eqref{Main equation with obstacle}.
The following result is analogous to \cite[Theorem 1.2]{LLLS2019partial} and \cite[Theorem 1.6]{KU2019partial}.

\begin{thm}[Simultaneous recovery: Unknown obstacle and coefficient]\label{Main Thm 4}
	Let $\Omega \subset \R^n$, $n\geq 2$ be a bounded connected domain with connected $C^\infty$ boundary $\p \Omega$. Let $D_1, D_2\Subset  \Omega$ be nonempty open subsets with $C^\infty$ boundaries such that $\Omega \setminus \overline{D_j}$ are connected. For $j=1,2$,
	let $a_j=a_j(x,y)$ be polyhomogeneous functions in $y\in \R$, for $x\in \overline{\Omega}\setminus D_j$.
	Denote by $\Lambda_{a_j ,D_j}^{\Gamma}$ the partial DN maps of the following Dirichlet problems 
	\begin{align*}
	\begin{cases}
	\Delta u_j +a_{j}(x,u_j)=0 & \text{ in }\Omega \setminus \overline{D_j}, \\
	u_j =0 & \text{ on }\p D_j,\\
	u_j =f & \text{ on }\p \Omega
	\end{cases}
	\end{align*}
	defined for any $f\in C^{2,\alpha}_0(\Gamma)$ with $\norm{f}_{C^{2,\alpha}_0(\Gamma)}<\delta$, where $\delta>0$ is a sufficiently small number.
	Assume that 
	\[
	\Lambda_{a_1, D_1}^{\Gamma}(f)= \Lambda_{a_2, D_2}^{\Gamma}(f), \text{ for any } \norm{f}_{C^{2,\alpha}_0(\Gamma)}<\delta. 
	\] 
	Then 
	\[
	D:=D_1 = D_2,
	\]
	and 
	\begin{align*}%\label{taylor_series_agree_cavity}
		b_{1,l}(x,y)  = b_{2,l}(x,y), \quad \text{ for } x\in \Omega\setminus\overline{D}, \ y\in \R \text{ and } l\in \N.
	\end{align*}
%	\begin{align}\label{taylor_series_agree_cavity}
%		\lim_{\eps_0\to 0} \LC \eps_0^{-\alpha_l}\p^{k_l}_y a_1(x,0) \RC = \lim_{\eps_0\to 0} \LC \eps_0^{-\alpha_l}\p^{k_l}_y a_2 (x, 0) \RC\text{ in }\Omega\setminus \overline{D} ,\quad \text{ for }l \in \N.
%	\end{align}
\end{thm}
\begin{rmk}
	It is worth emphasizing that the simultaneous recovery of an embedded obstacle and the surrounding potentials in the linear setting, for example, the linear Schr\"odinger equation (i.e., for the case $r=1$ in Theorem \ref{Main Thm 4}) is an open problem. We refer readers to \cite{isakov1990inverse,LLLS2019partial} for further discussions and \cite{CLL2017simultaneously} for arguments in a linear nonlocal setting.
\end{rmk}

The proof of Theorem \ref{Main Thm 4} is similar to the proof of Theorem \ref{Main Thm 3}, and the only difference is that we need to recover the unknown obstacle first. The method to recover the unknown obstacle has been investigated in \cite[Theorem 1.2]{LLLS2019partial}. We will give the proof in Section \ref{Sec 4}.

We are also able to extend the geometric results in \cite{LLLS2019nonlinear} to fractional power type nonlinearities. We refer to~\cite{LLLS2019nonlinear} for the introduction of these problems.
%~\f{These can be shortened into a more informal remark.}

\begin{thm}[Simultaneous recovery  of metric and potential in the plane]\label{Main Thm 5}
	Let  $(M_1,g_1)$ and $(M_2,g_2)$ be two compact connected $C^\infty$ Riemannian manifolds with mutual $C^\infty$ boundary $\p M$ and $\dim (M_1)=\dim (M_2) =2$. For $j=1,2$, let $\Lambda_{M_j,g_j,q_j}$ be the DN maps of 
	\begin{align}\label{2Deq}
	\Delta_{g_j} u +q_j \abs{u}^{r-1}u=0  \text{ in }M_j,
	\end{align}
	where  $r>1$ is a fractional number.
	%on $\Gamma \subset \p M$, for $j=1,2$.
	Let $0<\alpha<1$ and assume that 
	$$\Lambda_{M_1,g_1,q_1}  (f)=\Lambda_{M_2,g_2,q_2} (f) \text{ on }\p M, $$ for any $f\in C^{2,\alpha}(\p M)$ with $\norm{f}_{C^{2,\alpha}(\p M)}\leq \delta$, where $\delta >0$ is a sufficiently small number. Then:
	\begin{itemize}
		\item[(1)] There exists a conformal diffeomorphism $J:M_1\to M_2$ and a positive smooth function $\sigma\in C^\infty(M_1)$ such that 
		\begin{equation*}
		\sigma J^*g_2 = g_1 \text{ in }M_1,
		\end{equation*}
		with $J|_{\p M}=\Id$ and $\sigma|_{\p M} =1$.

		\item[(2)] Moreover, one can also recover the potential up to a natural gauge invariance in the sense that
		\[
		\sigma q_1 = q_2 \circ J \text{ in }M_1.
		\]
	\end{itemize}
	\end{thm}

Furthermore, as shown in \cite{LLLS2019nonlinear} for integer power type nonlinearities, one can also consider the corresponding Calder\'on type inverse problem on a transversally anisotropic manifold. Let us consider inverse problems for the semilinear Schr\"odinger equation on \emph{transversally anisotropic} manifold with fractional power type nonlinearities. The definition of a transversally anisotropic manifold is given as follows.

\begin{defi}
	Let $(M,g)$ be a compact oriented manifold with a $C^\infty$ boundary and with $\dim M\geq 3$.
	%\begin{itemize}
	%\item[(1)] 
	$(M,g)$ is called \emph{transversally anisotropic} if $(M,g)\Subset   (T,g)$, where $T=\R \times M_0$ and $g(x)=g(x_1,x')=e(x_1)\oplus g_0(x')$ for $x_1\in \R$ and $x'\in M_0$. Here $(\R, e)$ denotes the Euclidean line and $(M_0, g_0)$ stands for an $(n-1)$-dimensional compact manifold with a smooth boundary. 
	
	%	\item[(2)] $(M,g)$ is called conformally transversally anisotropic (CTA) if $(M,cg)$ is transversally anisotropic for some smooth positive function $c: M\to \R_+$.
	
	%	\end{itemize}
\end{defi}

	\begin{thm}\label{Main Thm 6}
		Let $(M,g)$ be a transversally anisotropic manifold, let $q_j \in C^{\infty}(M)$, and let $\Lambda_{q_j}$ be the DN maps for the equations 
		\[
		\Delta_g u + q_j \abs{u}^{r-1}u =0 \text{ in }M
		\]
		for $j=1,2$, where we further assume the fractional number satisfies
		\[
		r> 3. 
		\]
		Suppose that the DN maps satisfy 
		\[
		\Lambda_{q_1}(f) =\Lambda_{q_2}(f) \text{ on }\p M,
		\]
		for all $f$ with $\norm{f}_{C^{2,\alpha}(\p M)}\leq \delta$, for a sufficiently small number $\delta>0$ and for some $0<\alpha<1$. Then $q_1=q_2$ in $M$.
	\end{thm}
	
	Theorems \ref{Main Thm 5} and \ref{Main Thm 6} follow from the corresponding arguments in \cite{LLLS2019nonlinear} if we use the integral identity~\eqref{dm_lambdaq_identity} with the choice $v_0=1$ in $M$ (by taking $f_0=1$ on $\p M$).

The structure of this article is given as follows. In Section \ref{Sec 2}, we give well-posedness results for the relevant semilinear elliptic equations and derive the integral identity which plays a crucial role in the study of our inverse problems. In Section \ref{Sec 3}, we prove global uniqueness and simultaneous recovery in the Euclidean case, i.e., Theorems \ref{Main Thm 1}-\ref{Main Thm 4}. Finally, we prove Theorems \ref{Main Thm 5}-\ref{Main Thm 6} in Section \ref{Sec 4}.

\section{Preliminaries}\label{Sec 2}

First, let us recall the definition of H\"older spaces. Let $U\subset\R^n$ be an open set, let $k\in \N \cup \{0\}$, and let $0<\alpha <1$. The function space  $C^{k,\alpha}(\ol{U})$ consists of those real valued functions $u \in C^k(\ol{U})$ for which the norm 
\[
\norm{f}_{C^{k,\alpha}(\ol{U})}:=\sum_{|\gamma|\leq k}\norm{\p ^\gamma f}_{L^\infty(U)}+\sup_{x\neq y, \ x,y\in \ol{U}}\sum_{|\gamma|=k}\frac{|\p ^\gamma f(x)-\p^\gamma f(y)|}{|x-y|^\alpha},
\]
is finite. Here $\gamma=(\gamma_1,\cdots,\gamma_n)$ is a multi-index with $\gamma_i \in \N \cup \{0\}$ and $|\gamma|=\gamma_1 +\cdots +\gamma_n$.  
Furthermore, we also denote the space
\[
C_0^{k,\alpha}(\ol{U}):=\text{closure of }C^\infty_c(U) \text{ in }C^{k,\alpha}(\ol{U}).
\]
In short, we only use $C^\alpha(\ol{U})$ to denote $C^{0,\alpha}(\ol{U})$ when $k=0$. In addition, one can define H\"older spaces on any Riemannian manifold $(M,g)$ using the Riemannian distance or via local coordinates, see e.g. \cite[Section 13.8 in vol.\ III]{taylor2011partial}. %We next show the well-posedness of some semilinear elliptic equations for small data.

\subsection{Well-posedness}
Let $(M,g)$ be a $C^\infty$  compact Riemannian manifold with $C^\infty$-smooth boundary $\p M$. 
We study the well-posedness of the following boundary value problem 
\begin{align}\label{Dirichlet problem for well-posedness}
	\begin{cases}
		\Delta _g u +a(x,u)=0 &\text{ in } M,\\
		u=f &\text{ on }\p M,
	\end{cases}
\end{align}
for any sufficiently small Dirichlet data $f\in C^{2,\alpha}(\p M)$, for some $0<\alpha<1$.
Let us assume that the nonlinear coefficient $a=a(x,y) \in C^{k,\alpha}_{\mathrm{loc}}(\mR, C^{\alpha}(M))$ for some $k \geq 1$, meaning that $y \mapsto \p_y^j a(\,\cdot\,, y)$ is a continuous map $\mR \to C^{\alpha}(M)$ for $0 \leq j \leq k$ and for any $R > 0$, $\norm{\p_y^k a(\,\cdot\,,y) - \p_y^k a(\,\cdot\,,z)}_{C^{\alpha}} \leq C_R \abs{y-z}^{\alpha}$ whenever $\abs{y}, \abs{z} \leq R$. Also assume that the following two conditions hold:
\begin{gather}
	a(x,0) = 0, \quad \text{ for }x\in M,\label{a_first_condition} \\
	\text{The map $v \mapsto \Delta_g v + \p_y a(\,\cdot\,,0)v$ is injective on $H^1_0(M)$.} \label{a_second_condition}
\end{gather}
%~\footnote{Isn't $\p_ua(x,0)=0$?}
We prove the well-posedness of \eqref{Dirichlet problem for well-posedness} for small Dirichlet data $f\in C^{2,\alpha}(\p M)$. 
%The well-posedness of the Dirichlet problem \eqref{Dirichlet problem for well-posedness} is demonstrated as follows.

\begin{prop}[Well-posedness]\label{Prop: wellposedness_and_expansion}
	Let $(M,g)$ be a compact Riemannian manifold with $C^\infty$ boundary $\p M$ and let $Q$ be the semilinear elliptic operator 
	\[
	Q(u):= \Delta_g u + a(x,u),
	\]
	where $a \in C^{k,\alpha}_{\mathrm{loc}}(\mR, C^{\alpha}(M))$ for some $k \geq 1$, $\alpha \in (0,1)$, and \eqref{a_first_condition} and \eqref{a_second_condition} are satisfied.
	There exist $\delta, C > 0$ such that for any $f$ in the set 
	\[
	U_{\delta} := \left\{ h \in C^{2,\alpha}(\p M) \,;\, \norm{h}_{C^{2,\alpha}(\p M)} \leq \delta \right\},
	\]
    there is a solution $u = u_f$ of 
	\begin{align}\label{solvability}
	\begin{cases}
	\Delta_g u+a(x,u)=0 & \text{ in } M,\\
	u= f & \text{ on } \p M,
	\end{cases}
	\end{align}
	which satisfies
	\begin{align}\label{stability estimate in well-posedness}
	\norm{u}_{C^{2,\alpha}(M)}\leq C \norm{f}_{C^{2,\alpha}(\p M)}.
	\end{align}
	The solution $u_f$ is unique within the class $\left\{ w \in C^{2,\alpha}(M) \,;\, \norm{w}_{C^{2,\alpha}(M)} \leq C \delta \right\}$.
	%and if {\color{red}$f \in C^{\infty}(\p M)$, then $u_f \in C^{\infty}(M)$}~\footnote{not true}. 
	In addition, there are $C^{k}$ Frech\'et differentiable maps 
%	\begin{align*}
%	\begin{array}{rll}
%	S:& \!\!\!U_{\delta} \to C^{2,\alpha}(M), \ & f \mapsto u_f, \\
%	\Lambda:& \!\!\!U_{\delta} \to C^{1,\alpha}(\p M), \ &f \mapsto \p_{\nu} u_f|_{\p M}.
%	\end{array}
%	\end{align*}
	\begin{align*}
	&S: U_{\delta} \to C^{2,\alpha}(M), \quad f \mapsto u_f, \\
	&\Lambda: U_{\delta} \to C^{1,\alpha}(\p M), \quad f \mapsto \p_{\nu} u_f|_{\p M}.
	\end{align*}
	%Especially, if $s=2+\alpha$, then these maps $S$ and $\Lambda$ are $C^r$ Frech\'et differentiable.~\footnote{Make it clear that $r>s-2$.}
\end{prop}

In particular, if $a(x,u)=q(x)|u|^{r-1}u$ for a fractional number $r>1$ and $q\in C^\alpha(M)$, then the function $q(x)|u|^{r-1}u$ satisfies the condition $a(x,0)=\p _y a(x,0)=0$,
%~\footnote{Refers to "3rd" condition after \eqref{general coefficient a(x,z) 2}. Wouldn't the first condition \eqref{general coefficient a(x,z) 1} suffice, since above $\p_y a(x,0)=0$ so it's just Laplace equation which has unique solution? -TT} 
which implies that the conditions \eqref{a_first_condition} and \eqref{a_second_condition} hold automatically (due to the well-posedness of the Laplace equation). Hence, Proposition \ref{Prop: wellposedness_and_expansion} implies the well-posedness of the Dirichlet problem \eqref{Main equation} immediately.

For the proof of Proposition \ref{Prop: wellposedness_and_expansion}, we will need a lemma that will also be useful later.

\begin{lem} \label{lemma_b_calpha}
Let $(M,g)$ be a compact Riemannian manifold with $C^\infty$ boundary $\p M$, let $0 < \alpha < 1$, and let $b(x,y) \in C^{\alpha}_{\mathrm{loc}}(\mR, C^{\alpha}(M))$. For any $u \in C^1(M)$ one has $b(x, u(x)) \in C^{\alpha}(M)$, and 
\begin{equation} \label{dkua_holder_estimate}
\norm{b(x, u+v) - b(x,u)}_{C^{\alpha}(M)} = o(1), \text{ as $\norm{v}_{C^1(M)} \to 0$}.
\end{equation}
\end{lem}
\begin{proof}
The assumption that $t \mapsto b(\,\cdot\,,t)$ is a $C^{\alpha}_{\mathrm{loc}}$ function $\mR \to C^{\alpha}(M)$ means that for any $R > 0$ there is $C_R > 0$ such that 
	\begin{align*}
	\abs{b(x,t)} &\leq C_R, \\
	\abs{b(x,t)-b(y,t)} &\leq C_R d_g(x,y)^{\alpha}, \\
	\abs{b(x,t) - b(x,s)} &\leq C_R \abs{t-s}^{\alpha}, \\
	\abs{b(x,t) - b(x,s) - (b(y,t) - b(y,s))}&\leq \norm{b(\ccdot,t)-b(\ccdot,s)}_{C^\alpha(M)}d_g(x,y)^{\alpha} \\
	&\leq C_R d_g(x,y)^{\alpha} \abs{t-s}^{\alpha},
	\end{align*}
whenever $x, y \in M$ and $\abs{t}, \abs{s} \leq R$. 

Now if $u \in C^1(M)$ with $\norm{u}_{L^{\infty}(M)} \leq R$, one has $\abs{b(x,u(x))} \leq C_R$ and 
\begin{align*}
\abs{b(x,u(x))-b(y,u(y))} &\leq \abs{b(x,u(x)) - b(y,u(x))} + \abs{b(y,u(x)) - b(y,u(y))} \\
 &\leq C_R \left[ 1 + \norm{u}_{C^1(M)}^{\alpha} \right] d_g(x,y)^{\alpha}.
\end{align*}
This shows that $b(x,u(x)) \in C^{\alpha}(M)$.

	Let now $u, v \in C^1(M)$ with $\norm{u}_{L^{\infty}} \leq R$ and $\norm{u+v}_{L^{\infty}} \leq R$. Then
	\[
	\norm{b(x,u+v)-b(x,u)}_{L^{\infty}(M)} \leq C_R \norm{v}_{L^{\infty}(M)}^{\alpha}.
	\]
	Let us next estimate the $C^{\alpha}$ norm of $b(x,u+v)-b(x,u)$. Writing $h(x,u) :=  b(x,u)$ and $w_t(x) := u(x) + tv(x)$, we have 
	\begin{align}\label{estimate of h function in Sec 2}
    \begin{split}
    		&\abs{h(x,w_1(x))-h(x,w_0(x)) - [h(y,w_1(y)) - h(y,w_0(y))]} \\
    	\leq &\abs{h(x,w_1(x))-h(x,w_0(x)) - [h(y,w_1(x))-h(y,w_0(x))]} \\
    	& + \abs{h(y,w_1(x)) - h(y,w_0(x)) - [h(y,w_1(y)) - h(y,w_0(y))]}.
    \end{split}
	\end{align}
	The first absolute value on the right of \eqref{estimate of h function in Sec 2} is $\leq C_R d_g(x,y)^{\alpha} \abs{v(x)}^{\alpha}$. The second absolute value on the right of \eqref{estimate of h function in Sec 2} can be estimated by grouping the terms in two different ways and using the triangle inequality: it is either $\leq C_R \norm{v}_{L^{\infty}(M)}^{\alpha}$ or $\leq C_R \LC \norm{u}_{C^1(M)} + \norm{v}_{C^1(M)}\RC^{\alpha} d_g(x,y)^{\alpha}$.

	By interpolation, this shows that for any $\beta < \alpha$ one has 
	\[
	\norm{b(x,u+v)-b(x,u)}_{C^{\beta}(M)} = o(1),   \text{ as }\norm{v}_{C^1(M)} \to 0.
	\]
	This estimate is also true for $\beta = \alpha$. This can be seen by writing 
	\[
	b = b_{\eps} + r_{\eps},
	\] 
	where 
	\[
	b_{\eps}(x,t) = \int_{\R} \varphi_{\eps}(t-s) b(x,s) \,ds.
	\] 
	Here $\varphi_{\eps}(t) = \eps^{-n} \varphi(t/\eps)$ is a standard mollifier with $\varphi \in C^{\infty}_c((-1,1))$, $0 \leq \varphi \leq 1$, and $\int_{\R} \varphi(t) \,dt = 1$. Repeating the argument above for $b_{\eps}$ using a higher H\"older exponent in $t$, and using the estimate $\norm{r_{\eps}(\,\cdot\,,t)}_{C^{\alpha}(M)} \leq C_R \eps^{\alpha}$ for $\abs{t} \leq R$ which follows from the regularity of $b$, finally yields the estimate 
	\[
	\norm{b(x,u+v)-b(x,u)}_{C^{\alpha}(M)} = o(1), \text{ as $\norm{v}_{C^1(M)} \to 0$.} \qedhere
	\]
\end{proof}

\begin{proof}[Proof of Proposition \ref{Prop: wellposedness_and_expansion}]
	%Since $\p_u H(x,u)=0$, the wellposedness of~\eqref{solvability} directly follows from \cite[Chapter 14]{taylor2011partial-3}.
	We prove the existence of solutions by using the implicit function theorem in Banach spaces~\cite[Theorem 4.B]{Zeidler1986}.
	%~\cite[Theorem 10.6]{renardy2006introduction}.
 Let 
	\[
	 X=C^{2,\alpha}(\p M), \quad Y=C^{2,\alpha}(M), \quad Z=C^{\alpha}(M)\times C^{2,\alpha}(\p M).
	\]
	Consider the map 
	\[
	 F:X\times Y\to Z, \quad F(f,u)=\LC Q(u),u|_{\p M}-f \RC.
	\]
	Now $F$ indeed maps to $Z$, since by Lemma \ref{lemma_b_calpha} the map $u \mapsto a(x,u)$ takes $C^{2,\alpha}(M)$ to $C^{\alpha}(M)$. Thus $F$ is well defined.
	 
	 We next show that $F$ is a $C^k$ map. Let $0<m\leq k$ be an integer. If $u, v \in C^{2,\alpha}(M)$ we use the Taylor formula  
	\begin{align}\label{Taylor formula in Section 2}
	\begin{split}
		&a(x,u+v) \\
		=& \sum_{j=0}^{m-1} \frac{\p_u^j a(x,u)}{j!} v^j + \int_0^1 \frac{\p_u^{m} a(x,u+tv)}{(m-1)!} v^{m} (1-t)^{m-1} \,dt \\
		=&\sum_{j=0}^{m} \frac{\p_u^j a(x,u)}{j!} v^j-\frac{v^m}{m!}\p_u^m a(x,u)+\int_0^1 \frac{\p_u^{m} a(x,u+tv)}{(m-1)!} v^{m} (1-t)^{m-1} \,dt \\
		=&\sum_{j=0}^{m} \frac{\p_u^j a(x,u)}{j!} v^j+\frac{v^m}{(m-1)!}\int_0^1 \left[\p_u^{m} a(x,u+tv)-\p_u^{m} a(x,u) \right](1-t)^{m-1}\,dt.
	\end{split}
	\end{align}
	We study the remainder term. From \eqref{dkua_holder_estimate} with $b = \p_u^m a$ we obtain the estimate 
	\[
	\norm{\p_u^{m} a(x,u+tv)-\p_u^{m} a(x,u)}_{C^{\alpha}(M)} = o(1), \text{ if $t \in [0,1]$ and $\norm{v}_{C^{2,\alpha}(M)} \to 0$.}
	\]
	Inserting this in the Taylor formula computation \eqref{Taylor formula in Section 2} yields 
%
	%If $m=k$, then in the integral $\p_u^{m+1}a(x,u(x)+tv(x))$ needs to be justified as an $C^s$ function in $x$.)
% 	We have that~\f{I don't know how to prove this, but it is analogous to the fact that derivative of absolutely continuous function is in $L^1$.}
% 	\[
% 	 \left \|\int_0^1 \p_u^{m+1} a(x,u+tv)\right \|_{C^{s-2}(M)}\leq C_{m,u},
% 	\]
	%since $r>k+s-2\geq m+s-2$. 
	%Consequently, we obtain
	\[
	\left \| a(x,u+v) - \sum_{j=0}^m \frac{\p_u^j a(x,u)}{j!} v^j\right\|_{C^{\alpha}(M)} = o\LC \norm{v}_{C^{2,\alpha}(M)}^{m}\RC , \text{ as $\norm{v}_{C^{2,\alpha}(M)} \to 0$.}
	\]
% 	
% 	Here we also used that $a(x,u)$ is $C^{s-2}$ in the $x$ variable sinve $r>s-2$ and $\norm{u}_{C^s(M)}, \norm{v}_{C^s(M)} \leq 1$. The bound follows since the $k+1$ derivative of $a(x,\ccdot)\in C^r$ 
% 
% 	
% 	Since $C^s(M)$ is an algebra, we have that when $\norm{v}_{C^s(M)} \leq 1$ one has 
% 	\[
% 	\left \| a(x,u+v) - \sum_{j=0}^m \frac{\p_u^j a(x,u)}{j!} v^j\right\|_{C^{s-2}(M)} \leq C_{m,u} \norm{\int_0^1 \frac{\p_u^{m+1} a(x,u+tv)}{m!}}_{C^{s-2}(M)} \norm{v}_{C^s(M)}^{m+1}.
% 	\]
	This shows that $u \mapsto a(x,u)$ is a $C^k$ map $C^{2,\alpha}(M) \to C^{\alpha}(M)$. Since the other parts of $F$ are linear, $F$ is a $C^k$ map. % in the standard sense of~\cite[Definition 10.2]{renardy2006introduction}.~\f{We need the Holder version of Fr\'echet differentiability.}
    
    Note that $F(0,0) = 0$ by \eqref{a_first_condition}. The linearization of $F$ at $(0,0)$ in the $u$-variable is 
    %\footnote{Typo? should it be $a(x,0)$ here} 
    \[
    \left. D_uF\right|_{(0,0)}(v) = \LC \Delta_g v + \p_u a(x,0) v,v|_{\p M} \RC.
    \]
    This is a homeomorphism $Y\to Z$ by \eqref{a_second_condition}. 
    To see this, let $(w,\phi)\in Z=C^{\alpha}(M)\times C^{2,\alpha}(\p M)$, and consider the Dirichlet problem 
    \begin{align}\label{first linearization equation in the well-posedness}
    \begin{cases}
    (\Delta_g + \p_u a(x,0))v=w & \text{ in } M, \\
    v=\phi & \text{ on } \p M.
    \end{cases}
    \end{align}
The solution of \eqref{first linearization equation in the well-posedness}, if it exists, is unique by \eqref{a_second_condition}, and by using the Fredholm alternative and Schauder estimates the solution $v\in Y=C^{2,\alpha}(M)$ exists (see e.g.\ \cite[Exercise 1 in Section 13.8]{taylor2011partial}) and depends continuously on the data $(w,\phi)$. %Therefore, $\left. D_uF^{-1}\right|_{(0,0)}$ is a continuous linear map. 
  Thus the implicit function theorem in Banach spaces \cite[Theorem 4.B]{Zeidler1986} %H\"older version of the implicit function theorem in Banach spaces~\cite[Theorem B.4]{Zeidler1986} and \cite[proof of Corollary A.4]{Eldering2013}
    %Thus the implicit function theorem in Banach spaces~\cite[Theorem 10.6 and Remark 10.5]{renardy2006introduction}~\f{We need the Holder version of implicit function theorem.} 
    yields that there is $\delta>0$, a closed ball $U_{\delta}=\ol{B_X(0,\delta)}\subset X$, and a $C^k$ map $S: U \to Y$ such that whenever $\norm{f}_{C^{2,\alpha}(\p M)} \leq \delta$ we have 
    \[
     F(f,S(f))=(0,0).
    \]
    Since $S$ is Lipschitz continuous and $S(0) = 0$, $u = S(f)$ satisfies 
    \[
     \norm{u}_{C^{2,\alpha}(M)}\leq C \norm{f}_{C^{2,\alpha}(\p M)}.
    \]
    Moreover, by redefining $\delta$ if necessary $u=S(f)$ is the only solution to $F(f,u)=(0,0)$ whenever $\norm{u}_{C^{2,\alpha}(M)}\leq C \delta$.
	We have proven the existence of unique small solutions of the Dirichlet problem~\eqref{solvability} and the fact that the solution operator $S: U_{\delta} \to C^{2,\alpha}(M)$ is a $C^{k}$ map. Since the normal derivative is a linear map $C^{2,\alpha}(M) \to C^{1,\alpha}(\p M)$, it follows that also $\Lambda$ is a well defined $C^{k}$ map $U_{\delta} \to C^{1,\alpha}(\p M)$.	
	%Finally if $s=2+\alpha$, the condition $r\geq k+s-2$ is valid as an identity and we have that $S$ and $\Lambda$ are $C^r$ (H\"older) Frech\'et differentiable.
\end{proof}

In the next proposition we present an integral identity involving the $k$th linearization the DN map $\Lambda_q$. Below, we write 
\[
(D^k f)_x(y_1,\ldots,y_k) %= f^{(k)}(x; y_1, \ldots, y_k)
\]
to denote the $k$th derivative at $x$ of a $C^k$ map $f$ between Banach spaces, considered as a symmetric $k$-linear form acting on $(y_1,\ldots,y_k)$. We refer to  \cite[Section 1.1]{hormander1983analysis}, where the notation $f^{(k)}(x; y_1, \ldots, y_k)$ is used instead of $(D^k f)_x(y_1,\ldots,y_k)$. %This notation should probably be explained in the introduction.]}

%For $f \in C^{2,\alpha}(\p M)$ for some $0<\alpha<1$, let us denote by $v_f$ the unique solution of the Laplace equation 
%\begin{equation*}
%	\Delta_g v_f = 0 \text{ in $M$}, \quad \text{ and } \quad  v_f|_{\p M} = f.
%\end{equation*}

%With this notation at hand, we have the following proposition.
%
%We write the statement of the proposition on a Riemannian manifold $(M,g)$ for future reference and in case the reader is interested in applying these results for equations on manifolds. The proposition can be used in the $\R^n$ case by setting $g=\mathrm{Id}_{n\times n}$ and taking $M=\Omega\subset\R^n$. The proof stays unchanged.

\begin{prop}[Integral identity] \label{Prop: derivs_and_integral_formula}
	Let $(M,g)$ be a compact $C^\infty$ Riemannian manifold with a $C^\infty$ smooth boundary $\p M$. Let $q \in C^{\alpha}(M)$, and let $\Lambda_q$ be the DN map for the semilinear elliptic equation 
	\begin{equation}\label{Dir_prblm}
		\Delta_g u + q \abs{u}^{r-1}u = 0 \text{ in $M$},
	\end{equation}
	where
	\[
	r=k+\alpha, \quad k \geq 1 \text{ and } \alpha \in (0,1).
	\]
	Let $f_0\in C^{2,\alpha}(\p M)$. Then the $k$th linearization $(D^k \Lambda_q)_{\eps_0 f_0}$ of $\Lambda_q$ at $\eps_0 f_0$ satisfies the following identity:  
	For any $f_1, \ldots, f_{k+1} \in C^{2,\alpha}(\p M)$ one has 
	\begin{align} \label{dm_lambdaq_identity}
	\begin{split}
			&\lim_{\eps_0\to 0} \eps_0^{-\alpha} \int_{\p M} (D^k \Lambda_q)_{\eps_0 f_0}(f_1, \ldots, f_k)f_{k+1}  \,dS \\
			& \quad = c_r \int_M  q \abs{v_{0}}^{r-1}v_0^{1-k} v_{1}\cdots v_{k+1} \,dV,
	\end{split}
	\end{align}
	where $c_r$ is the constant given by 
	\begin{align*}%\label{constant c_r}
	c_r=-r(r-1)\cdots (r-(k-1)).
	\end{align*}

	Here each $v_{\ell}$, $\ell =0,\ldots,k+1$, is a harmonic function satisfying 
	\begin{align}\label{first linearized harmonic on manifold}
		\begin{cases}
			\Delta_g v_{\ell}=0 & \text{ in } M,\\
			v_{\ell}=f_\ell &\text{ on } \p M.
		\end{cases}
	\end{align}
	
\end{prop}
\begin{proof}
	Let $f_0\in C^{2,\alpha}(\p M)$ and denote $h_0=\eps_0f_0$, where $\eps_0$ is small. The nonlinearity $a(x,u) = q(x) |u|^{r-1}u$ satisfies the conditions in Proposition \ref{Prop: wellposedness_and_expansion}, and thus the DN map $\Lambda_q = \p_{\nu} S|_{\p M}$ is well defined for boundary data $f$ with $\norm{f}_{C^{2,\alpha}(\p M)} \leq \delta$. Here $S: f \mapsto u_f$ is the solution operator for the Dirichlet problem of the equation~\eqref{Dir_prblm}. 
	
	We first compute the derivatives of $\Lambda_q$ at $h_0$. For this it is enough to consider the derivatives of $S$. Let us write 
	\[
	\wt f = \wt f(x;\eps_1,\ldots, \eps_k):= \eps_1 f_1 (x)+ \ldots + \eps_k f_k(x).
	\]
	%f =f(\eps_1,\ldots, \eps_k):= \eps_1 f_1 + \cdots + \eps_k f_k$. 
	Let $f=h_0 +\wt f$, then the solution  
	\[
	u_{f}:=S(f)=S(h_0+\eps_1 f_1 + \cdots + \eps_k f_k)\in C^{2,\alpha}(M)
	\]
	is $k$ times continuously  differentiable with respect to the parameters $\eps_1, \ldots, \eps_k$ 
	%since $S:U_\delta\to C^{s}(M)$ is $k$ times continuosly Frech\'et differentiable 
	%is $C^r$ $C^{2,\alpha}(\p M)\to C^{2,\alpha}(M)$~\f{Put the small neighborhood in domain.}  
	by Proposition \ref{Prop: wellposedness_and_expansion}. %Here $U_\delta$ is as in Proposition~\ref{wellposedness_and_expansion}. 
	Let us denote
	\[
	\eps:=(\eps_0,\eps'), \quad \eps':=(\eps_1,\ldots,\eps_k).
	\]
	
	Applying $\left.\p_{\eps_1} \cdots \p_{\eps_j}\right|_{\eps'=0}$ to the Taylor formula for $C^k$ maps (see e.g.\ \cite[equation (1.1.8)]{hormander1983analysis}) 
	\[
	u_f = S(h_0 + \wt f) = \sum_{m=0}^k \frac{(D^m S)_{h_0}(\wt f, \ldots, \wt f)}{m!} + o\LC \lVert \wt f \rVert_{C^{2,\alpha}(\p M)}^k\RC
	\]
	implies that $(D^m S)_{h_0}$ for $0 \leq m \leq k$ may be computed using the formula 
	\begin{equation} \label{djs_formula}
	(D^m S)_{h_0}(f_1, \ldots, f_m) = \left. \p_{\eps_1} \cdots \p_{\eps_m} u_f \right|_{\eps'=0}.
	\end{equation}
	Moreover, since $S$ is $C^{k}$ map $C^{2,\alpha}(\p M)\to C^{2,\alpha}(M)$, since $u \mapsto q(x) \abs{u}^{r-1} u$ is a $C^k$ map $C^{2,\alpha}(M) \to C^{\alpha}(M)$ by the argument in Proposition \ref{Prop: wellposedness_and_expansion}, and since $\Delta_g$ is linear, we may differentiate the equation 
	\begin{equation} \label{mth_power_non-linearity}
		\begin{cases}
			\Delta_g u_f + q(x) \abs{u_f}^{r-1}u_f = 0 &\text{ in }M,\\
			u_f =f= h_0+\wt f  &\text{ on }\p M,
		\end{cases}
	\end{equation}
	up to $k$ times in the $\eps_\ell$ variables at $\eps'=0$ (recalling that $\wt f=f(x;\eps')=\wt f(x;\eps_1,\ldots,\eps_k)$).
	
	%\textbf{Case $1<r<2$:} 
	%When $1<r=1+\<2$, we have that $S$ is H\"older continuous. 
	%
	%\textbf{Case $r>2$.}
	Let $\ell\in \{1,\ldots, k\}$. Then for any $\beta>0$ we have the identity 
	\begin{align*}
		\p_{\eps_\ell} \LC \abs{u_f}^{\beta}u_f\RC = (\beta\abs{u_f}^{\beta-2}u_f^2 + \abs{u_f}^\beta)\p_{\eps_\ell} u_f 
		 =(\beta +1 )|u_f|^{\beta} \p _{\eps_{\ell}}u_f
	\end{align*}
    so that
	\begin{align}\label{first_lin}
		\begin{cases}
			\Delta_g \LC\p_{\eps_\ell} u_{f}|_{\eps' = 0}\RC + q(x) r|u_f|^{r-1}  \p_{\eps_\ell }u_f \big|_{\eps'=0} = 0 &\text{ in }M,\\
			\left. \p_{\eps_\ell} u_f \right|_{\eps' = 0} = f_\ell &\text{ on }\p M.
		\end{cases}
	\end{align}		
	Thus the first linearization of the map $S$ at $h_0$ is 
\begin{align}\label{eq:first linearization of S}
	v^{\eps_0}_{\ell}:= (DS)_{h_0}(f_\ell) = \left. \p_{\eps_\ell} u_{f}\right|_{\eps'=0} 
\end{align}
	where $v_{\ell}^{\eps_0}$ satisfies \eqref{first_lin}. For $\ell=1,2,\ldots, k$, we also claim that 
\begin{equation}\label{v ell epsilon limit}
	\lim_{\eps_0\to 0} v_{\ell}^{\eps_0}=v_{\ell} \text{ in } C^{2,\alpha}(M),
\end{equation}
	where $v_{\ell}$ is the harmonic function satisfying \eqref{first linearized harmonic on manifold} with Dirichlet data $f_\ell$.
	To prove \eqref{v ell epsilon limit}, note by the Schauder estimates we have % for $\Delta_g$ we have
	\begin{align*}
		\norm{v_{\ell}^{\eps_0}-v_\ell}_{C^{2,\alpha}(M)}&\leq  C\LC \norm{\Delta_g(v_\ell^{\eps_0}-v_{\ell})}_{C^{\alpha}(M)}+\norm{\eps_0f_0+f_\ell-f_\ell}_{C^{2,\alpha}(\p M)} \RC \\
		&= C\left(
\norm{q \left[r\abs{u_f}^{r-1}\p_{\eps_{\ell}} u_f\right] \big|_{\eps'=0}}_{C^{\alpha}(M)}
		+\norm{\eps_0f_0}_{C^{2,\alpha}(\p M)} \right) \\
		&\leq C \left( \norm{\abs{u_{\eps_0 f_0}}^{r-1}}_{C^{\alpha}(M)} + \eps_0 \right).
	\end{align*}
	Now $\norm{u_{\eps_0 f_0}}_{C^{2,\alpha}(M)} \leq C \eps_0 \norm{f_0}_{C^{2,\alpha}(\p M)}$ by \eqref{stability estimate in well-posedness}. Then \eqref{dkua_holder_estimate} with $b(x,t)$ replaced by $\abs{t}^{r-1}$ implies that $ \norm{\abs{u_{\eps_0 f_0}}^{r-1}}_{C^{\alpha}(M)} \to 0$ as $\eps_0\to 0$, proving \eqref{v ell epsilon limit}.
	
	Let now $2 \leq j \leq k$. Applying $\left. \p_{\eps_1} \cdots \p_{\eps_j}\right|_{\eps'=0}$ to \eqref{mth_power_non-linearity} gives that 
	\begin{align*}
		\begin{cases}
			\Delta_g \LC  \left. \p_{\eps_1} \cdots \p_{\eps_j} u_f\right|_{\eps'=0}\RC = %\left. H_k \LC u,\p_{\eps_\ell} u, \ldots, \p_{\eps'}^k u \RC \right|_{\eps'=0} 
- \left. \p_{\eps_1}\cdots \p_{\eps_j} \LC q(x)|u|^{r-1}u \RC \right|_{\eps'=0}
			 &\text{ in }M, \\
			\left. \p_{\eps_1} \cdots \p_{\eps_j} u_f \right|_{\eps'=0} = 0 &\text{ on }\p M,
		\end{cases}
	\end{align*}
%	where 
%	\begin{align*}
%		H_k \LC u,\p_{\eps_\ell} u, \ldots, \p_{\eps'}^k u \RC=\left. \p_{\eps_1}\cdots \p_{\eps_k} \LC q(x)|u|^{r-1}u \RC \right|_{\eps'=0}
%	\end{align*}
%	is a polynomial containing partial derivatives of $u$ with respect to $\eps_\ell$-variables, for $\ell =1,2,\ldots, k$.
	%~\f{Sketch an induction argument why this is true.}  
	%whose coefficients are all homogeneous in $u$ of orders greater than $\alpha$, and continuous in its other arguments. 
	%
	%
	%This can be seen from the identity $\p_{\eps_l} \abs{u}^\gamma=\gamma \abs{u}^{\gamma-2}u\p_{\eps_l} u$ holding for any $\gamma>1$.
	Since $r > k$, the fact that $u_f$ is $k$ times continuously Frech\'et differentiable in $\eps'$ gives that 
	\[
	\lim_{\eps_0\to 0}
	%\left. H_k\LC u,\p_{\eps_\ell} u, \ldots, \p^k_{\eps'} u\RC \right|_{\eps'=0} 
	\left. \p_{\eps_1}\cdots \p_{\eps_j} \LC q(x)|u|^{r-1}u \RC \right|_{\eps'=0}
	=0.
	\]
	By an argument similar to the one above using Schauder estimates we obtain
	%\[
	% \lim_{\eps_1\to 0}\p_{\eps_1} \cdots \p_{\eps_k} u_f|_{\eps'=0}\Big]|_{\p M}= \lim_{\eps_0\to 0}\eps_1^{1-\alpha}f_1=0.
	%\]
	%Thus
	\[
	\lim_{\eps_0\to 0}\left. \p_{\eps_1} \cdots \p_{\eps_j} u_f\right|_{\eps'=0}=0.
	\]
	%by the uniqueness of the Dirichlet problem.
	
%	{\color{red} YH: This integral identity might depend on the integer $k$ is even or odd...
%	
%%	TT: For example, for $k=2$ we would have
%%	$$
%%		\p_{\eps_\ell} \LC \abs{u_f}^{\beta}u_f\RC 
%%		 =(\beta +1 )|u_f|^{\beta} \p _{\eps_{\ell}}u_f
%%	$$
%%	and
%%		$$
%%		\p_{\eps_j}\p_{\eps_\ell} \LC \abs{u_f}^{\beta}u_f\RC 
%%		 =(\beta +1 )\beta|u_f|^{\beta-2}{\color{blue}u_f} \p_{\eps_j}u_f\p _{\eps_{\ell}}u_f
%%	$$
Let us consider the $k$th mixed derivative $w^{\eps_0} :=  \p_{\eps_1} \cdots \p_{\eps_k} u_f|_{\eps'=0}$ further. It satisfies the equation 
	\begin{align} \label{weps_equation}
		\begin{cases}
			\Delta_g w^{\eps_0} = %\left. H_k \LC u,\p_{\eps_\ell} u, \ldots, \p_{\eps'}^k u \RC \right|_{\eps'=0} 
- \left. \p_{\eps_1}\cdots \p_{\eps_k} \LC q(x)|u|^{r-1}u \RC \right|_{\eps'=0}
			 &\text{ in }M, \\
			w^{\eps_0} = 0 &\text{ on }\p M,
		\end{cases}
	\end{align}
	We wish to multiply \eqref{weps_equation} by $\eps_0^{-\alpha}$ and take the limit as $\eps_0 \to 0$. Since $f(t) = \abs{t}^{r-1} t$ for $r=k+\alpha$ satisfies the homogeneity relation $f(\lambda t) = \lambda^r f(t)$ for $\lambda > 0$, we have that
\begin{equation*}
\frac{d^k}{d y^k} \LC |y|^{r-1}y \RC= r(r-1)\cdots (r-(k-1)) |y|^{r-1} y^{1-k} = -c_r |y|^{r-1} y^{1-k}.
\end{equation*}
%Therefore, by the chain rule, the mixed derivative $\left.\p_{\eps_1} \cdots \p_{\eps_k} q \LC \abs{u_f}^{r-1}u_f\RC \right|_{\eps'=0}$ is
%	\[
%	r(r-1)\cdots(r-k)q\abs{u_f|_{\eps'=0}}^{r-1}u_f|_{\eps'=0}^{1-k} \LC\p_{\eps_1}u_f|_{\eps'=0} \RC\cdots \LC \p_{\eps_k}u_f|_{\eps'=0}\RC.
%	\]
%	}	

Using Fa\`a di Bruno's formula, see \cite{Hardy2006combinatorics}, we find that 
\begin{equation}\label{eq:Bruno}
\begin{split}
	\left.\p_{\eps_1} \cdots \p_{\eps_k} \LC \abs{u_f}^{r-1}u_f\RC \right|_{\eps'=0} 
	=&
	\sum_{\sigma\in P} c_\sigma |u_f|^{r-1}u_f^{1-|\sigma|}
	\prod_{\delta\in \sigma} \p^\delta_{\eps'} u_f\Big|_{\eps'=0}\\
	=&
	c_r \abs{u_f}^{r-1}u_f^{1-k} \LC\p_{\eps_1}u_f \RC\cdots \LC \p_{\eps_k}u_f\RC|_{\eps'=0} \\
	&+\sum_{\substack{\sigma\in P,\\ |\sigma|<k}} c_\sigma |u_f|^{r-1}u_f^{1-|\sigma|}
	\prod_{\delta\in \sigma}\p^\delta_{\eps'} u_f\Big|_{\eps'=0},
\end{split}
\end{equation}
where $P$ contains all partitions of $\{1,\ldots,k\}$ and the product over $\delta\in\sigma$ runs over all sets in the partition $\sigma$. The number $|\sigma|$  denotes the cardinality of the set $\sigma$ and  $\p^\delta_{\eps'}$ is the usual multi-index notation for partial derivatives in $\eps'$.

Observe that $u_f|_{\eps'=0}$ solves the nonlinear equation~\eqref{Dir_prblm} with boundary value $ h_0 =\eps_0f_0$. By continuity and uniqueness of solutions, we have that
\begin{equation}\label{v0 epsilon limit}
\eps_0^{-1}u_f \big|_{\eps'=0}\to v_{0} \text{ in } C^{2,\alpha}(M), \quad \text{ as }\eps_0 \to 0.
\end{equation}
%
%The H\"older space $C^{\alpha}(M)$ is an algebra due to  
%\begin{align}\label{Holder algebra}
%	\begin{split}
%		\norm{ \phi \psi }_{C^{\alpha}(M)}
%		\leq  C\left(\norm{\phi}_{C^{\alpha}(M)}\norm{\psi}_{L^\infty(M)}+\norm{\phi}_{L^\infty(M)} \norm{\psi}_{C^{\alpha}(M)} \right),
%	\end{split}
%\end{align}
%for any $\phi , \psi\in C^\alpha(M)$.
%
Then note that $|\sigma|<k$ implies that the products
$$
\prod_{\delta\in \sigma}\p^\delta_{\eps'} u_f\Big|_{\eps'=0}
$$
are bounded in $C^{\alpha}(M)$ as $\eps_0\to 0$, because the solution operator $S$ is continuously $k$-Fr\'echet differentiable and the H\"older space $C^\alpha(M)$ is an algebra.
Next, since the function $g(y) = |y|^{r-1}y^{1-|\sigma|}$ is homogeneous of degree $k-|\sigma|+\alpha\geq 1+\alpha$,  Euler's homogeneous function theorem shows that it belongs to $C^{1}(\R)$.
Since the composition of $C^1(\R)$ function with a $C^{2,\alpha}(M)$ function is at least $C^\alpha(M)$, we have that
\begin{equation}\label{holder_composition}
\left.\eps^{-\alpha} |u_f|^{r-1}u_f^{1-|\sigma|}\right|_{\eps'=0} = \left. \eps_0^{k-|\sigma|}
\left|\frac{u_f}{\eps_0}\right|^{r-1}\left(\frac{u_f}{\eps_0}\right)^{1-|\sigma|}\right|_{\eps'=0} \to 0 \quad\text{in } C^{\alpha}(M)
\end{equation}
as $\eps_0\to 0$. By using \eqref{v ell epsilon limit}, \eqref{v0 epsilon limit} and \eqref{holder_composition}, we see that after multiplying~\eqref{eq:Bruno} by $\eps_0^{-\alpha}$ and taking the limit $\eps_0\to 0$, only the first term on the right hand side of~\eqref{eq:Bruno} survives.
To analyze this first term in the right-hand side of~\eqref{eq:Bruno}, observe that $g(y)=|y|^{r-1}y^{1-k}$ belongs to $C^\alpha(\R)$ and $u_f$ is in $C^{2,\alpha}(M)$, so the composition $|u_f|^{r-1}u_f^{1-k}$ is in $C^\alpha(M)$. 
Recall again from~\eqref{eq:first linearization of S} that $\left. \p_{\eps_\ell} u_{f}\right|_{\eps'=0}\to v_\ell$ in $C^{2,\alpha}(M)$ as $\eps_0\to 0$ for all $\ell=1,2,\ldots, k$.
Due to the continuity of the solution map $S$, we finally have in $C^\alpha$ the limit
	\begin{equation}\label{rhs_dk_limit}
		\lim_{\eps_0\to 0}\eps_0^{-\alpha}\left. \p_{\eps_1} \cdots \p_{\eps_k}  \LC  q \abs{u_f}^{r-1}u_f\RC \right|_{\eps'=0} =-c_r q \abs{v_0}^{r-1} v_0^{1-k} v_1 \cdots v_k.
	\end{equation}

	Integrating the equation \eqref{weps_equation} against the harmonic function $v_{k+1}$, we have 
	\[
	\int_{\p M} (\p_{\nu} w^{\eps_0}) f_{k+1} \,dS = - \int_M \p_{\eps_1} \cdots \left. \p_{\eps_k} \LC  q(x)\abs{u_f}^{r-1}u_f\RC \right|_{\eps'=0} v_{k+1} \,dV.
	\]
	Since $\Lambda_q = \p_{\nu} S$ where $\p_{\nu}$ is linear, the formula \eqref{djs_formula} gives that $\p_{\nu} w^{\eps_0}|_{\p M} = (D^k \Lambda_q)_{\eps_0 f_0}(f_1, \ldots, f_k)$. Now \eqref{rhs_dk_limit} yields 
	\[
	\lim_{\eps_0 \to 0} \eps_0^{-\alpha} \int_{\p M} (D^k \Lambda_q)_{\eps_0 f_0}(f_1, \ldots, f_k) f_{k+1} \,dS = c_r \int_M q \abs{v_{0}}^{r-1}v_{0}^{1-k} v_{1} \cdots v_{k} \,dV
	\]
	as required.
		%On the other hand, since $\p_{\eps_1} \cdots \p_{\eps_k} \LC q(x) |u_f|^{r-1}u_f  \RC $ is a sum of terms containing positive powers of $u_f$, which are equal to zero when $f = 0$. 
	%Uniqueness of solutions for the Laplace equation implies that 
	%\[ (D^l S)_0(f_1, \ldots, f_l) = 0, \quad  \text{ for } 2 \leq l \leq k, \]
	%which yields that $\LC D^l \Lambda_q  \RC_0 =\left. \p _\nu \LC D^l S \RC_0 \right|_{\p M}=0$, for $2\leq l \leq k$. In particular, by linearity and continuity, one must have
	%\[\lim_{\eps_0\to 0}\eps_0^{-\alpha}(D^k \Lambda_q)_{h_0} = \left. \p_{\nu} \LC \lim_{\eps_0\to 0}\eps_0^{-\alpha} D^kS \RC _0 \right|_{\p M}.\]
	%Consequently, integrating~\eqref{laplace_ur_equation} against a harmonic function $v_{k+1}$ with boundary value $f_{k+1}$ we obtain the identity~\eqref{dm_lambdaq_identity} in the claim. This concludes the proof.
\end{proof}

It is easy to see that the integral identity also holds for any $f\in C^{2,\alpha}_0(\Gamma)$, for any open subset $\Gamma \subset \p M$.
The following result is an easy consequence of the preceding proposition. For simplicity we only state the result in Euclidean domains.

\begin{cor}[Integral identity with partial data]\label{Cor: derivs_and_integral_formula}
	Let $\Omega \subset \R^n$ be a bounded domain with $C^\infty$-smooth boundary $\p \Omega$, for $n \geq 2$,  and let $\Gamma \subset \p \Omega$ be a nonempty relatively open subset.
	Let $q \in C^{\alpha}(\overline{\Omega})$ for some $0<\alpha<1$, and let $\Lambda_q^\Gamma$ be the partial data DN map for the semilinear elliptic equation 
	\begin{equation*}%\label{BVP of integral formula}
		\begin{cases}
			\Delta u + q|u|^{r-1}u=0 & \text{ in }\Omega,\\
			u=f & \text{ on }\p \Omega,
		\end{cases}
	\end{equation*}
	where $r = k + \alpha$ with $k \geq 1$ and $\alpha \in (0,1)$. The $k$th linearization $D^k \Lambda_q^\Gamma$ of $\Lambda_q^\Gamma$ satisfies the following identity:  
	For any $f_0, f_1, \ldots, f_{k+1} \in C^{2,\alpha}_0(\Gamma)$, one has 
	\begin{align} \label{dm_lambdaq_identity in Euclidean}
	\begin{split}
			& \lim_{\eps_0\to 0} \int_{\p \Omega} \eps_0^{-\alpha}\LC D^k \Lambda_q^\Gamma\RC _{\eps_0 f_0}(f_1, \ldots, f_{k})f_{k+1}  \,dS \\
			 & \quad = c_r \int_\Omega  q\abs{v_{0}}^{r-1}v_0^{1-k} v_{1}\cdots v_{k+1} \,dx,
	\end{split}
	\end{align}
	where $c_r=- r(r-1)\cdots (r-(k-1))$.
	Here each $v_{\ell}$, $\ell=0,\ldots,k+1$, is a harmonic function satisfying 
	\begin{align*}%\label{harmonic in the euclidean}
	\Delta v_{\ell}=0 \text{ in } \Omega \quad  \text{ and }\quad   v_{\ell}=f_\ell\text{ on } \p \Omega.
	\end{align*}
\end{cor}

The result follows immediately from Proposition \ref{Prop: derivs_and_integral_formula}, even if the Dirichlet data is supported in a relatively open subset $\Gamma \subset\p \Omega$.

%In Proposition \ref{Prop: derivs_and_integral_formula}, in case the number $r\in \N$, a similar integral identity has been investigated in \cite{LLLS2019nonlinear}.
%Via Proposition \ref{Prop: derivs_and_integral_formula}, one can immediately derive the following integral identity in the Euclidean space, whenever the Dirichlet data is supported in a relatively open subset $\Gamma \subset \p \Omega$.
%The integral identity is useful when dealing with  related partial data inverse problems.

It is worth mentioning that even in the case $1<r<2$ we can use \emph{two} boundary functions $f_0$ and $f_1$. A suitable choice of the Dirichlet data $f_0$ allows us to get rid of the nonlinear term $|v_0|^\alpha$, if necessary, while still retaining the ability to choose $f_1$ and the auxiliary function $f_2$ in an appropriate way.

\begin{rmk}
We mention that for nonlinearities $a(x,u) = q(x) \abs{u}^{\alpha} u$ where $q \in C^{\alpha}(M)$ and $\alpha \in (0,1)$, one can prove that the solution of 
	\begin{align*}
	\begin{cases}
	\Delta u_\eps + q|u_\eps|^{\alpha}u_\eps=0 & \text{ in } M, \\
	u_\eps=\eps f & \text{ on }\p M,
	\end{cases}
	\end{align*}
	where $f \in C^{2,\alpha}(\p M)$ and $\eps > 0$ is small, has the asymptotic expansion 
	\begin{align*}%\label{ansatz}
			u_\eps=\eps v + \eps^{1+\alpha} w+ O(\eps^{1+2\alpha}),
	\end{align*}
	where $v$ is the harmonic function satisfying 
	\begin{align*}%\label{harmonic function v}
		\begin{cases}
		\Delta v=0 &\text{ in }M, \\
		v=f &\text{ on }\p M, 
		\end{cases}
	\end{align*}
	and $w$ is the solution of 
	\begin{align*}%\label{equation for w}
		\begin{cases}
		\Delta w = - q|v|^{\alpha} v &\text{ in }M , \\
		w=0 &\text{ on }\p M.
		\end{cases}
	\end{align*}
	One could use such one-parameter asymptotic expansions to give alternative proofs of some of our full data inverse problems. However, we will instead use Proposition \ref{Prop: derivs_and_integral_formula} and Corollary \ref{Cor: derivs_and_integral_formula}, which are based on multiparameter expansions and will lead to more general results. For our proof of Theorem~\ref{Main Thm 6} it is crucial to use Proposition \ref{Prop: derivs_and_integral_formula} with $k\geq 3$.
\end{rmk}

\section{Global uniqueness in Euclidean space}\label{Sec 3}
In this section, let us prove our main Euclidean results. Recall that we are considering real-valued solutions. In order to apply the density results \cite{ferreira2009linearized,LLLS2019nonlinear} involving products of complex-valued harmonic functions, let us start with the following simple lemma also used in \cite{LLLS2019partial}:
\begin{lem} \label{lemma_complex_real_solutions}
Let $\Omega \subset \R^n$ be a bounded domain with $C^\infty$-smooth boundary $\p \Omega$, for $n \geq 2$. Let $f \in L^{\infty}(\Omega)$, $v_1, v_2 \in L^2(\Omega)$, and $v_3, \ldots, v_k \in L^{\infty}(\Omega)$ be complex valued functions where $k \geq 2$. Then 
\[
\int_{\Omega} f v_1 \cdots v_k \,dx = \sum_{j=1}^{2^k} \int_{\Omega} c_j f w_1^{(j)} \cdots w_k^{(j)} \,dx
\]
where $c_j\in \{\pm 1, \pm i \} $ and $w_1^{(j)}\in \{\mathrm{Re}(v_1), \mathrm{Im}(v_1) \}, \cdots, w_k^{(j)}\in \{\re(v_k), \im(v_k) \}$ for $1\leq j \leq 2^k$.
\end{lem}
\begin{proof}
The result follows by writing 
\[
\int_{M} f v_1 \cdots v_k \,dx = \int_{M} f (\re(v_1) + i \im(v_1)) \cdots (\re(v_k) + i \im(v_k)) \,dx 
\]
and by multiplying out the right hand side.
\end{proof}
Lemma~\ref{lemma_complex_real_solutions} also holds on Riemannian manifolds $(M,g)$, which will be applied in Section \ref{Sec 4}.

\begin{proof}[Proof of Theorem \ref{Main Thm 1}]
Since $\Lambda_{q_1}(f)=\Lambda_{q_2}(f)$ for all small $f$ and since $\Lambda_{q_j}$ is a $C^k$ map by Proposition \ref{Prop: wellposedness_and_expansion}, one has 
\[
\LC D^k \Lambda_{q_1}\RC _{\eps_0 f_0}(f_1, \ldots, f_k) = \LC D^k \Lambda_{q_2}\RC_{\eps_0 f_0}(f_1, \ldots, f_k)
\]
for all $f_0, \ldots, f_{k+1} \in C^{2,\alpha}(\p \Omega)$ and for $\eps_0$ small. The integral identity \eqref{dm_lambdaq_identity in Euclidean} applied with $q_1$ and $q_2$ implies that 
\[
\int_{\Omega} (q_1-q_2) \abs{v_0}^{r-1} v_0^{1-k} v_1 \cdots v_{k+1} \,dx = 0
\]
for any real-valued harmonic functions $v_0, \ldots, v_{k+1} \in C^{2,\alpha}(\ol{\Omega})$. Let $v_0=v_3 = \ldots = v_{k+1} = 1$ be constant functions in $\Omega$. Then 
\begin{equation} \label{product_pair_harmonic}
\int_{\Omega} (q_1-q_2) v_1 v_2 \,dx = 0
\end{equation}
whenever $v_j \in C^{2,\alpha}(\ol{\Omega})$ are real-valued and harmonic. Since the real and imaginary parts of a complex valued harmonic function are harmonic, it follows from Lemma \ref{lemma_complex_real_solutions} that \eqref{product_pair_harmonic} remains true for complex valued harmonic functions.

Now let $v_1(x)=e^{ (-\zeta+i\xi )\cdot x}$ and $v_2(x)=e^{(\zeta+i\xi )\cdot x}$ be Calder\'on's exponential solutions (see \cite{calderon}), which are harmonic, and where $\zeta,\xi \in \R^n$ with $|\zeta|=|\xi|$ and $\zeta\cdot \xi =0$. Then we have
	\begin{align}\label{Fourier transform of q_1 - q_2}
		\begin{split}
		0=&\int_{\Omega}(q_1-q_2)v_1 v_2 \, dx \\
		=& \int_{\Omega}(q_1-q_2)e^{(-\zeta+i\xi )\cdot x}e^{(\zeta+i\xi )\cdot x}\, dx \\
		=& \int_{\Omega}(q_1-q_2)e^{2i\xi \cdot x}\, dx.
		\end{split}
	\end{align}
	Thus, via \eqref{Fourier transform of q_1 - q_2}, we obtain that the Fourier transform of the difference $q_1-q_2$ at $-2\xi $ is zero. Since $\xi\in \R^n$ can be chosen arbitrarily, we must have $q_1=q_2$ as desired.
	
Let us give another proof of this result when $n \geq 3$ and when we only assume that $\Lambda_{q_1}(f)=\Lambda_{q_2}(f)$ for all small $f$ with $f \geq 0$. As before, let $f_0=f_3=\ldots=f_{k+1}=1$ so that $v_0=v_3=\ldots=v_{k+1}=1$ in $\Omega$. Then \eqref{product_pair_harmonic} holds whenever $f_1, f_2 \geq 0$. Let $x\not\in\overline{\Omega}$ and choose the boundary values $f_1,f_2$ so that $v_1(y)=v_2(y) = |x-y|^{2-n}$. Then $v_1,v_2>0$ are harmonic in $\Omega$.
Inserting these solutions to~\eqref{product_pair_harmonic} and writing $q = q_1-q_2$, we see that 
\begin{equation*}%\label{eq:id_positive}
\int_{\Omega} \abs{x-y}^{4-2n} q(y) \,dy = 0
\end{equation*}
for $x\not\in\overline{\Omega}$.
%
%Let us extend $q$ by zero outside $\overline{\Omega}$.
By \cite[page 79]{isakov1990inverse}, the knowledge of the Riesz potential 
\[
I_\beta \mu(x) = \int_{\Omega}|x-y|^{\beta} d\mu(y),
\]
for $x\not\in\overline{\Omega}$ uniquely determines the measure $\mu(y)$ in $\Omega$, when $\beta\neq 2k$ and $\beta+n\neq 2k+2$ for all $k=0,1,\ldots$. Since these conditions are satisfied for $\beta=4-2n$,  we see that $q=0$ by setting $d\mu(y)=q(y)\,dy$ above. Isakov~\cite{isakov1990inverse} credits M. Riesz~\cite{riesz1938} and M. M.~Lavrentiev~\cite{lavrentiev1967} for the first results about determination of a measure from the Riesz potential.
\end{proof}

	\begin{proof}[Proof of Theorem \ref{Main Thm 2}]
		Since the DN maps satisfy $\Lambda_{q_1}^\Gamma (f)= \Lambda_{q_2}^\Gamma(f)$ for any sufficiently small Dirichlet data $f\in C^{2,\alpha}_0(\Gamma)$, we have for any $f_0, \ldots, f_{k+1} \in C^{2,\alpha}_0(\Gamma)$ 
		\begin{align}\label{k-th DN agree in partial data}
			\lim_{\eps_0\to 0} \eps_0^{-\alpha} \int_{\p \Omega} \LC D^k \Lambda_{q_1}^\Gamma  -D^k \Lambda_{q_2}^\Gamma  \RC_{\eps_0 f_0} (f_1,\ldots, f_k) f_{k+1}\, dS=0.
		\end{align}
		Therefore, by subtracting the integral identity \eqref{dm_lambdaq_identity in Euclidean} for $q=q_1, q_2$ and inserting \eqref{k-th DN agree in partial data}, one has 
		\begin{align*}
			\int_{\Omega } (q_1-q_2)|v_{0}|^{r-1}v_0^{1-k} v_{1}\ldots v_{k+1}\, dx=0,
		\end{align*}
	where $v_{\ell}$ are the solutions of \eqref{first linearized harmonic on manifold} in $\Omega$ for $\ell=0,1,\ldots, k+1$ with $v_{\ell}|_{\p \Omega} = f_{\ell}$. Write $F := (q_1-q_2)|v_{0}|^{r-1}v_0^{1-k} v_3 \ldots v_{k+1}$, so that we have 
	\[
	\int_{\Omega} F v_1 v_2 \,dx = 0.
	\]
	By applying Lemma~\ref{lemma_complex_real_solutions}, we see that the last identity is valid for complex-valued harmonic functions $v_1, v_2 \in C^{2,\alpha}(\ol{\Omega})$ with $\supp(v_{\ell}|_{\p \Omega}) \subset \Gamma$. 
	On the other hand, via the density result of \cite{ferreira2009linearized}, one can choose $\left\{ v_{1}v_{2} \right\}$ to form a dense subset in $L^1(\Omega)$ with $\supp(v_{1}|_{\p \Omega}), \supp(v_{2}|_{\p \Omega}) \subset \Gamma$.  This implies that $F = 0$ in $\Omega$. Finally, by choosing  $f_0,f_3,\ldots ,f_{k+1} \not \equiv 0$ to be nonnegative Dirichlet data supported in $\Gamma$, we see that $v_0, v_3, \ldots, v_{k+1}$ are positive in $\Omega$ by the maximum principle. Thus one can conclude that $q_1 =q_2 $ in $\Omega$.
	\end{proof}

	Next we prove Theorem \ref{Main Thm 3}.
	
	\begin{proof}[Proof of Theorem \ref{Main Thm 3}]
		Via Proposition \ref{Prop: wellposedness_and_expansion}, 
		let $u_j\in C^{2,\alpha}(\overline{\Omega})$, for $j=1,2$, be the unique (small) solutions to 
		\begin{align}\label{equation of different parameters_thm13}
			\begin{cases}
				\Delta u_j +a_j(x,u_j) =0 &\text{ in }\Omega,\\
				u_j =\eps_0f_0+\eps_1 f_1  &\text{ on }\p \Omega,
			\end{cases}
		\end{align}
		where $\eps_\ell\geq 0$ are small parameters and $f_\ell \in C^{2,\alpha}_0(\Gamma)$, for $\ell =0,1$.
		Then, as in equation \eqref{eq:first linearization of S} in the proof of Proposition~\ref{Prop: derivs_and_integral_formula}, we have that 
		the first linearization of the solution map $S_j$ to~\eqref{equation of different parameters_thm13}, $j=1,2$, at $h_0:=\eps_0f_0$ satisfies 
\begin{equation*}%\label{first lin of S in thm 1.3}
		v^{\eps_0}_{j,1}:= (DS_j)_{h_0}(f_1) =\left. \p_{\eps_1} u_{j}\right|_{\eps_1=0} 
\end{equation*}
		where $v_{j,1}^{\eps_0}$ satisfies 
		\begin{align}\label{first_lin_thm13}
			\begin{cases}
				\Delta v^{\eps_0}_{j,1} =-\p_ya_j(x,\left. u_j\right|_{\eps_1=0}) v^{\eps_0}_{j,1} &\text{ in }\Omega,\\
				v^{\eps_0}_{j,1} = f_1 &\text{ on }\p \Omega,
			\end{cases}
		\end{align}		
		for $j=1,2$.
		Analogously to \eqref{v ell epsilon limit} in the proof of Proposition~\ref{Prop: derivs_and_integral_formula}, one has 
		\[
		v^{\eps_0}_{j,1}\to v_1 \text{ in } C^{2,\alpha}(\ol{\Omega}), \quad \text{ as }\eps_0 \to 0,
		\]
		where $v_{1}$ solves $\Delta v_1=0$ in $\Omega$ and $v_1|_{\p \Omega}=f_1$.
		
		Fix $f_2 \in C^{2,\alpha}_0(\Gamma)$ and let $v_2$ solve $\Delta v_2=0$ in $\Omega$ with $\left. v_2\right|_{\p\Omega}=f_2$. Since $\Lambda_{a_1}^\Gamma(f)=\Lambda_{a_2}^\Gamma(f)$ for any sufficiently small $f\in C^{2,\alpha}_0(\Gamma)$, integration by parts and \eqref{first_lin_thm13} yield that 
		\begin{equation}\label{Integral id for partial data}
		\begin{split}
			0=&\left.\p_{\eps_1}\right|_{\eps_1=0}\LC \int_{\p \Omega}f_2\LC \Lambda_{a_1}^\Gamma-\Lambda_{a_2}^\Gamma\RC \LC \eps_0f_0+\eps_1 f_1\RC\, dS \RC  \\
			=& \left. \p_{\eps_1}\right|_{\eps_1=0}\LC \int_\Omega v_{2}\LC\Delta u_1-\Delta u_2 \RC dx\RC + \left. \p_{\eps_1} \right|_{\eps_1=0}\LC \int_\Omega \nabla v_{2}\cdot\nabla \LC  u_1-u_2 \RC dx \RC \\
			=&-\int_\Omega v_{2} \left. \p_{\eps_1}\right|_{\eps_1=0}\LC a_1(x,u_1)-a_2(x,u_2)\RC\, dx\\
			&+\left. \p_{\eps_1}\right|_{\eps_1=0}\LC \int_{\p \Omega} \p_\nu v_{2} \LC u_1|_{\p \Omega}-u_2|_{\p \Omega}\RC \, dS \RC \\
			=&-\int_\Omega v_{2} \LC \p_ya_1(x,u_1|_{\eps_1=0}) v^{\eps_0}_{1,1}-\p_ya_2(x,u_2|_{\eps_1=0}) v^{\eps_0}_{2,1}\RC   dx\\
			&+\int_{\p \Omega} \p_\nu v_{2} \LC f_1-f_1\RC  dS\\
			=&-\int_\Omega v_{2} \LC \p_ya_1(x,u_1|_{\eps_1=0}) v^{\eps_0}_{1,1}-\p_ya_2(x,u_2|_{\eps_1=0}) v^{\eps_0}_{2,1}\RC  dx.
		\end{split}
		\end{equation}
		
		For $j=1,2$, the function 
		\[w_j := \left. u_j\right|_{\eps_1=0}
		\]
		now solves 
		\begin{align*}%\label{equation of different parameters_thm13}
			\begin{cases}
				\Delta w_j +a_j(x,w_j) =0 &\text{ in }\Omega,\\
				w_j =\eps_0f_0  &\text{ on }\p \Omega.
			\end{cases}
		\end{align*}
		By \eqref{stability estimate in well-posedness} we have 
		\[
		\norm{w_j}_{C^{2,\alpha}(\ol{\Omega})} \leq C \eps_0 \norm{f_0}_{C^{2,\alpha}(\p \Omega)}.
		\]
		Since $\Delta(w_j - \eps_0 v_0) = -a_j(x,w_j)$ in $\Omega$ with $w_j - \eps_0 v_0|_{\p \Omega} = 0$, Schauder estimates imply that 
		\[
		\norm{w_j-\eps_0v_0}_{C^{2,\alpha}(\ol{\Omega})} \leq C \norm{a_j(x,w_j)}_{C^{\alpha}(\ol{\Omega})}.
		\]
		Using the Taylor formula as in \eqref{Taylor formula in Section 2} together with the conditions $$
		a_j(x,0) = \p_y a_j(x,0) = 0
		$$ 
		gives that 
		\[
		a_j(x, w_j(x)) = w_j(x) \int_0^1 \LC \p_y a_j(x, tw_j(x)) - \p_y a_j(x,0)  \RC  \,dt.
		\]
		We may now apply \eqref{dkua_holder_estimate} with $b$ replaced by $a_j$ to obtain that 
\begin{equation}\label{u_j expansion}
		\begin{split}
		\norm{w_j-\eps_0 v_0}_{C^{2,\alpha}(\ol{\Omega})} &\leq C \norm{w_j}_{C^{\alpha}(\ol{\Omega})} \int_0^1 \left\|\p_y a_j(x, tw_j) - \p_y a_j(x,0)\right\|_{C^{\alpha}(\ol{\Omega})} \,dt \\
		&= o(\eps_0)
		\end{split}
\end{equation}
		as $\eps_0 \to 0$. 
		
		We have by assumption $a_j(x,y) \sim \displaystyle \sum_{l=1}^\infty b_{j,l} (x,y)$,
% 		
% 		$a_j(x,y)=\displaystyle\sum_{l=1}^\infty b_{j,l} (x,y)$, 
		where each $b_{j,l}(\ccdot,y)\in C^\alpha(\overline{\Omega})$ is homogeneous of order $r_l>1$ with respect to the variable $y\in \R$, for $l \geq 1$. Let us also write $\beta_{j,N} := a_j -\displaystyle \sum_{l=1}^{N-1} b_{j,l}$ for $j=1,2$ and $N \geq 1$, with $\beta_{j,1} = a_j$. Then $\beta_{j,N}$ is in $C^{1,\alpha}_{\mathrm{loc}}(\mR, C^{\alpha}(\ol{\Omega}))$ as in Definition~\ref{polyhomogeneous}. It follows from \eqref{beta_norm_bound} that, in particular, 
%We first make the following technical remark regarding the condition~\eqref{beta_norm_bound}. We have by choosing $N\geq 1$ for all $y\neq 0$ that
%		 \begin{align*}
% &\norm{y^{-1}\beta_{j,N}(\,\cdot\,,y)}_{C^{\alpha}(\ol{\Omega})} + \norm{\p_y \beta_{j,N}(\,\cdot\,,y)}_{C^{\alpha}(\ol{\Omega})} \\
% &\quad=\abs{y}^{-1}\norm{\beta_{j,N}(\,\cdot\,,y)}_{C^{\alpha}(\ol{\Omega})} + \norm{\p_y \beta_{j,N}(\,\cdot\,,y)}_{C^{\alpha}(\ol{\Omega})}\\
% &\quad\leq C_N\abs{y}^{r_N-1}. %\leq C_N \abs{y}^{r_N}, \qquad \abs{y} \leq 1.
% \end{align*}
% Since $\beta_{j,N}$ is in $C^{1,\alpha}_{\text{loc}}(\R,C^\alpha(\overline{\Omega}))$, the limit 
% \[
%  \lim_{y\to 0}\norm{y^{-1}\beta_{j,N}(\,\cdot\,,y)}_{C^{\alpha}(\ol{\Omega})}=\lim_{y\to 0}\norm{y^{-1}(\beta_{j,N}(\,\cdot\,,y)-\beta_{j,N}(\,\cdot\,,0))}_{C^{\alpha}(\ol{\Omega})}
% \]
%\lim_{y\to 0}\norm{y^{-1}\beta_{j,1}(\,\cdot\,,y)}_{C^{\alpha}(\ol{\Omega})}$ 
%exists (here we used~\eqref{beta_norm_bound} again to have $\beta_{j,N}(\,\cdot\,,0)=0$). We conclude that
% \[ \norm{\p_y \beta_{j,N}(\,\cdot\,,0)}_{C^{\alpha}(\ol{\Omega})}=0 \]
% and consequently we have for all $y\in \R$ (including $y=0$) that
 \begin{equation*}%\label{beta_estim_all_y}
  \left\Vert \p_y a_j(\,\cdot\,,y) - \sum_{l=1}^{N-1} \p_y b_{j,l}(\,\cdot\,,y) \right\Vert_{L^{\infty}(\Omega)}\leq C_N\abs{y}^{r_N-1}, \qquad \abs{y} \leq 1,
 \end{equation*}
for $j=1,2$.

 We apply the above with $N=2$ and $y=w_j(x) = u_j(x)|_{\eps_1=0}$ to have for $x\in \overline{\Omega}$, for $j=1,2$ that
 \[
		\left| \p_ya_j(x,w_j)-\p_yb_{j,1}(x,w_j)\right|\leq C_2 \left|w_j\right|^{r_2-1} \leq C \eps_0^{r_2-1}.
				\]
		Multiplying this by $\eps_0^{-r_1+1}$ and using the facts that $r_2 > r_1$ and $\p_yb_{j,1}(x,y)$ is homogeneous of order $r_1-1$ in $y$, we obtain in $L^{\infty}(\Omega)$ that
                \[
		\lim_{\eps_0\to 0} \eps_0^{-r_1+1}\p_ya_j(x,w_j)=\lim_{\eps_0\to 0}\p_yb_{j,1}(x,\eps_0^{-1}w_j)=\p_yb_{j,1}(x,v_0).
		\]
		Here in the last equality we additionally used~\eqref{u_j expansion}.
		Recall that we also have that the limit $\displaystyle \lim_{\eps_0\to 0} v_{j,1}^{\eps_0}= v_{1}$ in $C^{2,\alpha}(\ol{\Omega})$, for both $j=1,2$. Hence, we obtain
		\begin{equation*}%\label{eps limit recovers b}
		\begin{split}
			0&=\lim_{\eps_0\to 0}\eps_0^{-r_1+1}\int_\Omega v_{2} \left[ \p_ya_1(x,u_1|_{\eps_1=0}) v^{\eps_0}_{1,1}-\p_ya_2(x,u_2|_{\eps_1=0}) v^{\eps_0}_{2,1}\right] dx \\
			&=\int_\Omega \left[\p_yb_{1,1}(x,v_0)-\p_yb_{2,1}(x,v_0)\right] v_{1}v_{2}\, dx.
			\end{split}
		\end{equation*}
		Via the density result of~\cite{ferreira2009linearized}, products $v_1 v_2$ of pairs of harmonic functions with boundary values supported in $\Gamma\subset \p \Omega$ are dense in $L^1(\Omega)$. Therefore, we must have 
		\[
		\p_yb_{1,1}(x,v_0)=\p_yb_{2,1}(x,v_0), \text{ for }x \in \Omega.
		\]
		
		In addition, notice that the boundary value $f_0\in C_0^{2,\alpha}(\Gamma)$ has been arbitrary so far. Let $x_0\in \Omega$, let $y_0\in \R$ and let us choose by Runge approximation (see e.g.~\cite[Proposition A.2]{lassas2018poisson}) a boundary value $f_0=f_{0,x_0}\in C_0^\infty(\Gamma)$ so that 
		\begin{equation}\label{Runge_approx}
		v_0(x_0)=y_0.		
		\end{equation}
		We deduce that 
		\[
		\p_yb_{1,1}(x_0,y_0)=\p_yb_{2,1}(x_0,y_0)
		\]
		for any $x_0\in \Omega$ and any $y_0$. Thus we have $\p_yb_{1,1}=\p_yb_{2,1}$. By Euler's homogeneous function theorem, we have %~\f{Check this. Tony: This is ok, since $\p_{j,1}b(x,y)$ is $r_1$-homogeneous in $y$.}
		\begin{equation*}%\label{Euler_homog_trick}
		 b_{1,1}(x,y)=\frac{y}{r_1}\s \p_yb_{1,1}(x,y)=\frac{y}{r_1}\s \p_yb_{2,1}(x,y)=b_{2,1}(x,y),
		\end{equation*}
		where $r_1>1$ is the degree of homogeneity for $b_{j,1}(x,y)$ with respect to the $y$-variable, for $j=1,2$. Thus $b_{1,1} = b_{2,1}$.
		
		We proceed by induction on the index $l \in \N$ of $b_{j,l}$, $j=1,2$, to show that $b_{1,l}=b_{2,l}$ for any $l \in \N$. We have already shown the case $l=1$. Let us then make the induction assumption that $b_{1,l}=b_{2,l}$ for $l=1,\ldots,L$, for some $L\in \N$. Then, we have that 
		
		\begin{align*}
		&\abs{(\p_ya_1(x,y)-\p_ya_2(x,y))-(\p_yb_{1,L+1}(x,y)-\p_yb_{2,L+1}(x,y))} \\
		=&\Bigg|\big(\p_ya_1(x,y)-\p_ya_2(x,y)\big)-\sum_{l=1}^{L} \p_yb_{1,l}(x,y)+ \sum_{l=1}^{L} \p_yb_{2,l}(x,y)\\
		&\quad -\big(\p_yb_{1,L+1}(x,y)-\p_yb_{2,L+1}(x,y)\big)\Bigg|
		%\\ &-\p_ya_2(x,y)}-\abs{\p_yb_{1,L+1}(x,y)-\p_yb_{2,L+1}(x,y)}  \\ 
		\\
		= &\left|\LC \p_ya_1(x,y)-\sum_{l=1}^{L+1} \p_yb_{1,l}(x,y)\RC- \LC \p_ya_2(x,y)-\sum_{l=1}^{L+1} \p_yb_{2,l}(x,y) \RC \right| \\
		=&\left| \p_y\beta_{1,L+2}(x,y)- \p_y\beta_{2,L+2}(x,y) 
		%a_1(x,y)-\sum_{l=1}^{L+1} \p_yb_{1,l}(x,y)\RC- \LC \p_ya_2(x,y)-\sum_{l=1}^{L+1} \p_yb_{2,l}(x,y) 
		%
		 \right| 
			%&+ \sum_{l=1}^{L} b_{2,l}(x,y)+\p_yb_{2,L+1}(x,y) -\p_ya_2(x,y)} \\%+\abs{\p_yb_{1,l}(x,y)-\p_yb_{2,l}(x,y)} \\
		\leq 2\s  C_{L+2}\abs{y}^{r_{L+2}-1}. %+\abs{\p_yb_{1,l}(x,y)-\p_yb_{2,l}(x,y)}.
		\end{align*}
		Here we used the induction assumption in the first equality.
		Applying this for $y=w_j(x) = u_j(x)|_{\eps_1=0}$ we have for $x\in \overline{\Omega}$, and for $j=1,2$, that
		\[
		 \left| (\p_ya_1(x,w_j)-\p_ya_2(x,w_j))-(\p_yb_{1,L+1}(x,w_j)-\p_yb_{2,L+1}(x,w_j)) \right|\leq C \eps_0^{r_{L+2}-1},
		\]
		for some constant $C>0$.
		Here we used again $\norm{w_j}_{C^{2,\alpha}(\ol{\Omega})} \leq C \eps_0 \norm{f_0}_{C^{2,\alpha}(\p \Omega)}$

%%%		
%%%		{\color{gray}
%%%		\begin{align*}
%%%		&\abs{\p_ya_1(x,y)-\p_ya_2(x,y)}-\abs{\p_yb_{1,L+1}(x,y)-\p_yb_{2,L+1}(x,y)} \\
%%%		=&\abs{\p_ya_1(x,y)-\sum_{l=1}^{L} \p_yb_{1,l}(x,y)+ \sum_{l=1}^{L} \p_yb_{1,l}(x,y) \\ &-\p_ya_2(x,y)}-\abs{\p_yb_{1,L+1}(x,y)-\p_yb_{2,L+1}(x,y)}  \\
%%%		\leq &\abs{\p_ya_1(x,y)-\sum_{l=1}^{L} \p_yb_{1,l}(x,y)-\p_yb_{1,L+1}(x,y) \\
%%%			&+ \sum_{l=1}^{L} b_{2,l}(x,y)+\p_yb_{2,L+1}(x,y) -\p_ya_2(x,y)} \\%+\abs{\p_yb_{1,l}(x,y)-\p_yb_{2,l}(x,y)} \\
%%%		\leq & 2\s  C_{L+1}\abs{y}^{r_{L+1}}. %+\abs{\p_yb_{1,l}(x,y)-\p_yb_{2,l}(x,y)}.
%%%		\end{align*}
%%%		}
		%is a polyhomogeneous function in $y$, with degrees of homogeneity equal to and greater than $r_{L+1}-1$. 
		Therefore, by using~\eqref{u_j expansion}, homogeneity and  $r_{L+2}>r_{L+1}$, we obtain in $L^{\infty}(\Omega)$ that 
	     \begin{align*}
	     	& \lim_{\eps_0\to 0}\eps_0^{-r_{L+1}+1}\LC \p_ya_1(x,u_1|_{\eps_1=0})-\p_ya_2(x,u_2|_{\eps_1=0})\RC \\
	     	=&\lim_{\eps_0\to 0}\LC \p_yb_{1,L+1}(x,\eps_0^{-1}w_1)-\p_yb_{2,L+1}(x,\eps_0^{-1}w_2)\RC \\
	     	=&\p_yb_{1,L+1}(x,v_0)-\p_yb_{2,L+1}(x,v_0).
	     \end{align*}
	    By repeating the arguments we used to prove the special case $N=2$, which especially use the integral identity~\eqref{Integral id for partial data} and ~\cite{ferreira2009linearized}, we obtain 
		\[
		\p_yb_{1,L+1}=\p_yb_{2,L+1}.
		\]
		By Euler's homogeneous function theorem again, we then have $b_{1,L+1}=b_{2,L+1}$ in $\Omega$ as desired, which concludes the induction step and the proof of the theorem.
	\end{proof}		
\begin{rmk}
In the previous proof we recovered the expansion coefficients $b_{l}(x,y)$ of the potential $a\sim \sum_{l=1}^\infty b_l $ at arbitrary point $(x_0,y_0)\in \Omega\times \R$. This was done by using Runge approximation (see \eqref{Runge_approx}) to select a boundary value $f_0$ so that the corresponding solution $v_0$ satisfies $v_0(x_0)=y_0$. 
%Notice that the point $y_0\in \R$ is determined by the boundary data $f_0$ via the Runge approximation for the first linearized equation, see \eqref{Runge_approx}.
This is slightly different from earlier results in \cite{LLLS2019nonlinear,LLLS2019partial,KU2019partial}, where one recovers the Taylor coefficients $\tilde{b}_l(x,y):=\p^l_y \tilde{a}(x,y)$ of an unknown smooth potential $\tilde{a}(x,y)$ only at $y=0$, $x\in \Omega$.%~\footnote{Add some more citations? Note that KU19 recovers a potential $V(x,z)$ completely, but they assume a series expansion.}
\end{rmk}

In the end of this section, let us prove the simultaneous recovery of an obstacle and a potential.

\begin{proof}[Proof of Theorem \ref{Main Thm 4}]
For $\ell=0,1$, let $\eps_\ell\geq 0$ be sufficiently small parameters, and $f_\ell\in C^{2,\alpha}_0 (\Gamma)$. Consider the Dirichlet data $f=\eps_0f_0+\eps_1f_1$ and let $u_j=u_j(x)$ be the solution of
	\begin{align}\label{equation of different parameters_thm14}
		\begin{cases}
		\Delta u_j + a_j (x,u_j)=0 & \text{ in }\Omega ,\\ 
		u_j =0 & \text{ on }\p D_j, \\
		u_j =\displaystyle f &\text{ on }\p \Omega,
		\end{cases}
	\end{align}
	for $j=1,2$, where $a_j=a_j(x,z)$ are polyhomogeneous in the sense of Definition~\ref{polyhomogeneous} with $x \in \Omega \setminus \overline{D_j}$. We first show that $D_1=D_2$ and then recover the coefficients similarly as in the proof of Theorem \ref{Main Thm 3}.
	
	\vspace{3mm}
	
	{\it Step 1. Recovering the obstacle.}

    \vspace{3mm}
	
\noindent	
%Let us denote $\eps=(\eps_0,\eps')$, where $\eps'=(\eps_1,\eps_2)$. 
As in the proof of Proposition~\ref{Prop: derivs_and_integral_formula}, see~\eqref{eq:first linearization of S}, we have that the first linearization of the solution map $S_j$ to~\eqref{equation of different parameters_thm14}, $j=1,2$, at $h_0:=\eps_0f_0$ satisfies 
\begin{equation*}%\label{first lin of S in thm 1.3} 
		v^{\eps_0}_{j,\ell}:= (DS_j)_{h_0}(f_\ell) =\left. \p_{\eps_\ell} u_{j}\right|_{\eps'=0}, 
\end{equation*}
where $v^{\eps_0}_{j,\ell}$ is the solution of
		\begin{align*}
			\begin{cases}
				\Delta_g v^{\eps_0}_{j,\ell} = -\p_ya_j(x,\left. u_j\right|_{\eps'=0}) v^{\eps_0}_{j,\ell} &\text{ in }\Omega,\\
				v^{\eps_0}_{j,\ell} = 0 &\text{ on } \p D_j,\\
				v^{\eps_0}_{j,\ell} = f_\ell &\text{ on }\p \Omega.
			\end{cases}
		\end{align*}
		Analogously to \eqref{v ell epsilon limit} in the proof of Proposition~\ref{Prop: derivs_and_integral_formula}, one has 
% 		\[
% 		v^{\eps_0}_{j,1}\to v_1 \text{ in } C^{2,\alpha}(\ol{\Omega}), \quad \text{ as }\eps_0 \to 0,
% 		\]
% 		where $v_{1}$ solves $\Delta v_1=0$ in $\Omega$ and $v_1|_{\p \Omega}=f_1$.
% 		
% For $\ell=1,2,\ldots, k$, we also have
		\[
		v^{\eps_0}_{j,\ell}\to v_j^{(\ell)} \text{ in } C^{2,\alpha}(\ol{\Omega}\setminus D_j), \quad \text{ as }\eps_0 \to 0,
		\]
where
	\begin{align*}%\label{harmonic in the obstacle}
		\begin{cases}
		\Delta v_j^{(\ell)}=0 & \text{ in }\Omega \setminus \overline{D_j}, \\
		v_j^{(\ell)}=0 & \text{ on }\p D_j, \\
		v_j^{(\ell)}=f_\ell & \text{ on }\p \Omega
		\end{cases}
	\end{align*}
	for $j=1,2$ and $\ell =0,1$. The rest of the proof is the analogous to the proof of \cite[Theorem 1.2]{LLLS2019partial}. (See also \cite[Theorem 1.6]{KU2019partial}.) For the sake of completeness, we offer details of the proof below.
	
	 Let $G$ be the connected connected component of $\Omega\setminus (\overline{D_1 \cup D_2})$, whose boundary contains $\p \Omega$. Consider the function $\wt v^{(\ell)}:=v_1^{(\ell)}-v_2^{(\ell)}$, which solves%then $\wt v^{(\ell)}$ is the solution of 
	\begin{align*}
		\begin{cases}
		\Delta \wt v^{(\ell)}=0 &\text{ in }G,\\
		\wt v^{(\ell)}=\p_\nu \wt v^{(\ell)}=0 &\text{ on }\Gamma,
		\end{cases}
	\end{align*}
	where we have used that $\Lambda_{a_1,D_1}^\Gamma (f)=\Lambda_{a_2,D_2}^\Gamma (f)$, which holds for all sufficiently small Dirichlet data $f \in C^{2,\alpha}_0 (\Gamma)$. By the unique continuation of harmonic functions this yields that $\wt v^{(\ell)}=0$ in $G$. That is, for $\ell=0,1$, we have % which implies
	\begin{align}\label{equalsontilde}
		v_1^{(\ell)}=v_2^{(\ell)} \text{ in }G. %, \quad \text{ for }\ell=1,2.
	\end{align} 
	
	We use a contradiction argument to prove $D_1 =D_2$. For this, let us assume that $D_1\neq D_2$. Note that the connected component $G\neq \emptyset$. By using \cite[Lemma A.3]{LLLS2019partial}, there exists
	\[
	x_1 \in \p G \cap (\Omega \setminus \overline{D_1})\cap \p D_2.
	\]
    Since $x_1 \in \p D_2$, we have $v_2^{(\ell)}(x_1)=0$. By~\eqref{equalsontilde} and continuity, we also have that $v_1^{(\ell)}(x_1)=0$. Note that $x_1$ is an interior point of the open set $\Omega\setminus \overline{D_1}$. 
    
    We next fix one of the boundary values $f_\ell$ to be non-negative and not identically $0$.
	Since $v_1^{(\ell)}(x_1)=0$, the maximum principle implies that $v_1^{(\ell)}\equiv 0$ in $\Omega \setminus \overline{D_1}$, which contradicts to the assumption that
	$v_1^{(\ell)}=f_\ell$ on $\p \Omega$  is not identically zero (because the harmonic function $v_1^{(\ell)}$ is continuous up to boundary). This shows that  
	$$
	D:=D_1=D_2.
	$$ 
% 	Moreover, we have by~\eqref{equalsontilde} that 
% 	\begin{equation}\label{cavity lins}
% 	v^{(\ell)}:=v_1^{(\ell)}=v_2^{(\ell)} \text{ in } \Omega \setminus \overline{D}, \text{ for }\ell=1,2.
% 	\end{equation}
% 	
	
	{\it Step 2. Recovering the coefficient.}
	
	\vspace{3mm}
    
 \noindent Since we have proved that $D_1 = D_2 = D$, it follows that the partial data Dirichlet-to-Neumann maps for the equations $\Delta u + a_j(x,u) = 0$ in $\Omega \setminus \ol{D}$ agree on $\Gamma$. Applying Theorem \ref{Main Thm 3} in the connected set $\Omega \setminus \ol{D}$ then implies that $b_{1,l} = b_{2,l}$ for all $l\in \N$. This concludes the proof. 
\end{proof}
 
 %{\color{gray}~\f{There was a sketch of the proof, which is here in gray. I don't think we need it -Tony}
 %As in the proof of Theorem~\ref{Main Thm 3}, one first concludes by using the integral identity \eqref{eps limit recovers b} that
%		\begin{align*}
%			\int_{\Omega \setminus \overline{D}} \LC \p_yb_{1,1}(x,v_0)-\p_yb_{2,1}(x,v_0) \RC v_{1}v_{2}\, dx=0.
%		\end{align*}
%As the products of harmonic functions with boundary values supported in $\Gamma\subset \p \Omega$ are dense, one obtains $\p_y b_{1,1}(x,v_0) = \p_y b_{2,1}(x,v_0)$, $x\in \Omega\setminus \ol{D}$. One may use Runge approximation (\cite[Proposition A.2]{lassas2018poisson}) to find for any $x_0\in \Omega \setminus \ol{D}$ and $y_0\in \R$ a harmonic function $v_0$ such that $v_0(x_0)=y_0$. Together with Euler's homogeneous function theorem this implies that 
%$b_{1,1}(x,v_0) = b_{2,1}(x,v_0)$ and a choice of $v_0$ (by Runge approximation) such that $v_0(x_0)=y_0$ yields
%\[
%b_{1,1}(x_0,y_0) = b_{2,1}(x_0,y_0)
%\]
%for all $x_0\in\Omega\setminus \ol{D}$ and $y_0\in\R$.
%
%The claim that $b_{1,l} = b_{2,l}$ for all $l\in \N$ follows from an induction argument as in the proof of Theorem~\ref{Main Thm 3}.}
%Then, by using the polyhomogeneity of $a_j(x,y)$ for $j=1,2$, one can use the integral identity \eqref{Integral id for partial data} with $\eps_0^{r_\ell-1}$ instead of $\eps_0^{r_1-1}$ and induction to show that $b_{1,\ell}=b_{2,\ell}$, for all $\ell\in\N$.

\section{Global uniqueness in Riemannian manifolds}\label{Sec 4}

In this last section of this paper, we prove Theorem \ref{Main Thm 5} and Theorem \ref{Main Thm 6}. 
In our earlier work~\cite{LLLS2019nonlinear}, we proved similar theorems for power type nonlinearities, with integer exponents. 
%In this section, we demonstrate that these results hold also for the fractional power type nonlinearities. 
We begin with the proof of Theorem \ref{Main Thm 5}.
%~\f{I think we can do partial data as well. -Tony}

\begin{proof}[Proof of Theorem \ref{Main Thm 5}]
	The proof is similar to the proof of \cite[Theorem 1.2]{LLLS2019nonlinear}. We first recover the manifold and the its conformal class by the first linearization. After that we use the integral identity \eqref{dm_lambdaq_identity} to recover the potential.
	
		\vspace{3mm}
	
	{\it Step 1. Recovering the conformal manifold.}

	\vspace{3mm}
	
	\noindent  By using Proposition~\ref{Prop: wellposedness_and_expansion}, the equality $\Lambda_{M_1,g_1,q_1} (f) = \Lambda_{M_2,g_2,q_2}(f)$, for all $f\in C^{2,\alpha}(\p M)$ with $\norm{f}_{C^{2,\alpha}(\p M)}\leq \delta$, where $\delta>0$ is a sufficiently small number, implies
	\[
	\LC D\Lambda_{M_1,g_1,q_1} \RC_0 = \LC D\Lambda_{M_2,g_2,q_2} \RC_0.
	\]
	Here, for $j=1,2$, the maps $\LC D\Lambda_{M_j,g_j,q_j} \RC_0$ are the DN maps of the linearizations of the equations $\Delta_{g_j} u_j +q_j |u_j|^{r-1}u_j=0  \text{ in }M_j$ at a boundary value $f=0$. 
	This implies that the DN maps on $\p M$ of the first linearized equation 
	\begin{align*}
		\begin{cases}
			\Delta_{g_j} v_j =0 & \text{ in }M_j, \\
			v_j=f & \text{ on } \p M
		\end{cases}
	\end{align*}
	agree on $\p M$. That is, we know the DN maps on $\p M$ of the anisotropic Calder\'on problem on two-dimensional Riemannian manifolds. Thus, as noted in the proof of \cite[Theorem 1.2]{LLLS2019nonlinear}, we may use~\cite[Theorem 5.1]{lassas2018poisson} to determine the manifold and the Riemannian metric up to a conformal transformation:  There exists a $C^\infty$ smooth diffeomorphism $J:M_1 \to M_2$ such that 
	\begin{equation*}%\label{J_conformal}
		\sigma J^*g_2= g_1 \text{ in } M_1
	\end{equation*}
	with $J|_{\p M}=\Id$. Here the function $\sigma\in C^\infty(M_1)$ is positive with $\sigma |_{\p M} =1$.
	
		\vspace{3mm}
	
	{\it Step 2. Recovering the potential.}
	
	\vspace{3mm}
	
	\noindent 
	Let us transform the equation
	%\begin{align}\label{2Deq_proof}
	%
	$
	\Delta_{g_2} u_2 +q_2 \abs{u_2}^{r-1}u_2=0 
	$
	from the manifold $(M_2,g_2)$ into the manifold $(M_1,g_1)$ as follows. We denote in $M_1$
	\[
	\widetilde{q}_2=\sigma^{-1}(q_2\circ J) \equiv \sigma^{-1}J^*q_2.
	\]
	%holds.
	%(After that we will show by Colin-Leo paper that $q_1=\tilde{q}_2$.)
	Let $u_2$ be the solution to
	\begin{align}\label{equation of u_2 in 2D manifold}
		\begin{cases}
				\Delta_{g_2}u_2+q_2|u_2|^{r-1}u_2=0 & \text{ in }M_2, \\
				u_2=f &\text { on } \p M,
		\end{cases}
	\end{align}
	where $f\in  C^{2,\alpha}(\p M)$ with $\norm{f}_{C^{2,\alpha}(\p M)}\leq \delta$, $\delta>0$ sufficiently small. Let us define 
	\[
	\widetilde{u}_2:=
				J^*u_2\equiv u_2\circ J,
	\]
	in $M_1$. Then  $\widetilde{u}_2$ satisfies in $M_1$
	\begin{align*}
		&\Delta_{g_1}\widetilde{u}_2+\widetilde{q}_2|\widetilde{u}_2|^{r-1}\widetilde{u}_2\\
		=&\Delta_{\sigma J^*g_2}\widetilde{u}_2 +\widetilde{q}_2|\widetilde{u}_2|^{r-1}\widetilde{u}_2 \\
		=&\sigma ^{-1}\Delta_{J^*g_2}\widetilde{u}_2 +\sigma^{-1}(J^*q_2)|\widetilde{u}_2|^{r-1}\widetilde{u}_2\\
		=&\sigma ^{-1}J^*(\Delta_{g_2}u_2)+ \sigma^{-1}(J^*q_2)|J^*u_2|^{r-1}J^\ast u_2 \\
		=&\sigma^{-1}J^*\left(\Delta_{g_2} u_2  + q_2|u_2|^{r-1}u_2\right).
	\end{align*}
	Here we used the conformal invariance of the Laplace-Beltrami operator in two dimensions and the coordinate invariance of Laplace-Beltrami operator in the second and third equality respectively.
	Therefore, one has 
	\begin{align}\label{conformal change}
		\begin{cases}
			\Delta_{g_1}\widetilde{u}_2+\widetilde{q}_2|\widetilde{u}_2|^{r-1}\widetilde{u}_2=0 & \text{ in }M_1,\\
			\widetilde{u}_2=f & \text{ on }\p M,
		\end{cases}
	\end{align}
	where we have used that $u_2$ is the solution of \eqref{equation of u_2 in 2D manifold}, $f\in C^{2,\alpha}(\p M)$ and $J|_{\p M}=\text{Id}$.
	
	Let $u_1$ be the solution to the nonlinear equation $\Delta_{g_1} u_1 +q_1 \abs{u_1}^{r-1}u_1=0$
	in $M_1$ with potential $q_1$ and boundary data $f$. We show next that% the following equation 
	\begin{equation}\label{DN_maps_on_single_manifold}
		\p_{\nu_1} u_1=\p_{\nu_1}\widetilde{u}_2 \text{ on }\p M.
	\end{equation}
	Via the assumption that $\Lambda_{M_1,g_1,q_1}(f)=\Lambda_{M_2,g_2,q_2}(f)$, it follows that if $u_1 =u_2=f\in C^{2,\alpha}(\p M)$  on  $\p M$, then 
	\begin{equation}\label{DNmapsgive}
		\p_{\nu_1}u_1=\p_{\nu_2}u_2 \text{ on }\p M.
	\end{equation}
	We compute that 
	\begin{equation}\label{norm_calc}
		\p_{\nu_2}u_2=\nu_2\cdot du_2=\nu_2\cdot d(u_2\circ J \circ J^{-1})=(J^{-1}_*\nu_2)\cdot d\widetilde{u}_2=\nu_1\cdot d\widetilde{u}_2=\p_{\nu_1} \widetilde{u}_2,
	\end{equation}
	where $\cdot$ denotes the canonical pairing between vectors and covectors, and $d$ is the exterior derivative of a function. For example $\nu_2\cdot d u_2=g(\nu_2,\nabla u_2)=\sum_{k=1}^2 \nu_2^k\, \p_k u_2$.
	We used that $J:M_1 \to M_2$ is conformal diffeomorphism, $\sigma J^*g_2= g_1$, with $J|_{\p M}=\text{Id}$ and $\sigma|_{\p M} =1$ in \eqref{norm_calc}. Combining~\eqref{DNmapsgive} and~\eqref{norm_calc}, we have \eqref{DN_maps_on_single_manifold} as claimed. 
	
	We have by \eqref{DN_maps_on_single_manifold} that 
	\begin{equation}\label{transfd_DN_maps_agree}
		\Lambda_{M_1,g_1,q_1}(f) =\p_{\nu_1} u_1=\p_{\nu_1}\widetilde{u}_2 = \widetilde{\Lambda}_{M_1,g_1,\widetilde{q}_2}(f),
	\end{equation}
	for all $f\in C^{2,\alpha}(\p M)$ with $\norm{f}_{C^{2,\alpha}(\p M)}\leq \delta$, 
	where $\widetilde{\Lambda}_{M_1,g_1,\widetilde{q}_2}$ denotes the DN map of the Dirichlet problem \eqref{conformal change} on $\p M$.

	We apply Proposition~\ref{Prop: derivs_and_integral_formula} on $(M_1,g_1)$, the DN maps $\Lambda_{M_1,g_1,q_1}$ and $\widetilde{\Lambda}_{M_1,g_1,\widetilde{q}_2}$, which agree by~\eqref{transfd_DN_maps_agree}. By Proposition~\ref{Prop: wellposedness_and_expansion} we have 
	\[
	\lim_{\eps_0\to 0}\eps_0^{-\alpha}\left. \LC D^k \Lambda_{M_1,g_1,q_1}\RC \right|_{\eps_0f_0} =\lim_{\eps_0\to 0}\eps_0^{-\alpha} \left. \LC D^k \widetilde{\Lambda}_{M_1,g_1,\widetilde{q}_2}\RC \right|_{\eps_0f_0} \text{ on }\p M,
	\]
	and by Proposition~\ref{Prop: derivs_and_integral_formula}
	\[
	\int_{M_1} (q_1-\widetilde{q}_2) \abs{v_{0}}^{r-1}v_0^{1-k} v_{1}\cdots v_{k+1} \,dV = 0,
	\]
	% 
	% 
	%  \implies & (D^2 \Lambda_{M_1,g_1,q_1})_0 = (D^2 \widetilde{\Lambda}_{M_1,g_1,\tilde{q}_2})_0 \\
	%  \implies & \int_{M_1} (q_1-\widetilde{q}_2) v_1 v_2 v_3 \,dV = 0
	% \end{align*}
	where $v_0,\,v_1,\cdots ,v_k \in C^{2,\alpha}(M_1)$ are harmonic functions in $(M_1,g_1)$ with $r=k+\alpha>1$. We can choose $v_0 =v_1=\cdots=v_{k-2} =1$ in $M_1$, hence
	\[
	\int_{M_1} (q_1-\widetilde{q}_2) v_{k-1} v_k \,dV = 0
	\]
	for any harmonic functions $v_{k-1}$ and $v_k$ in $M_1$.
	
	%We choose $v_{k-1}$ and $v_k$ to be complex geometrical optics solutions constructed in \cite{guillarmou2011calderon} and 
	%This proves the assertion.	
	
	 By choosing $v_{k-1}$ and $v_k$ to be complex geometrical optics solutions constructed in \cite{guillarmou2011calderon} (see the proof of Proposition 5.1 in \cite{guillarmou2011calderon}), we conclude that 
	\[
	q_1 = \widetilde{q}_2 \text{ in }M_1.
	\]  
	We point out that the construction in~\cite{guillarmou2011calderon} can be simplified in our case where $v_{k-1}$ and $v_k$ are harmonic. In such case, Carleman estimates are not needed and the construction in \cite{guillarmousalotzou_complex} would suffice. 
	%By using \cite[Proposition 5.1]{guillarmou2011calderon}, one concludes that 
	%\[
	%q_1 = \widetilde{q}_2 \text{ in }M_1.
	%\]
	We have proven the claim. % the assertion.	
\end{proof}

\begin{proof}[Proof of Theorem \ref{Main Thm 6}]
		Let us write $r=k+\alpha$, $k\in \N$, $k\geq 3$ and $\alpha\in (0,1)$. For $j=1,2$, consider $\Lambda_{q_j}$ to be the DN map for the equation $\Delta_g u_j + q_j |u_j|^{r-1} u_j = 0$ in $M$. If $\Lambda_{q_1}(f) = \Lambda_{q_2}(f)$ for any sufficiently small $f\in C^{2,\alpha}(\p M)$, then by Proposition~\ref{Prop: wellposedness_and_expansion}
		\begin{align*}
			\lim_{\eps_0\to 0}\eps_0^{-\alpha}\Big( D^k \Lambda_{q_1}\Big)_{\eps_0f_0} =\lim_{\eps_0\to 0}\eps_0^{-\alpha} \Big( D^k \Lambda_{q_2} \Big)_{\eps_0f_0}. % \text{ on }\p M.
		\end{align*}
		 Hence, by Proposition \ref{Prop: derivs_and_integral_formula}, we have
	\[
	\int_M (q_1-q_2) \abs{v_{0}}^{r-1}v_0^{1-k} v_{1}\cdots v_{k+1}\,dV = 0,
	\]
	where $v_j \in C^{2,\alpha}(M)$ are harmonic functions in $M$. 
	%Since $r =k+\alpha> 3$, this implies that the integer $k\geq 3$.
	Therefore, by choosing $v_0\equiv 1$ and by using \cite[Proposition 5.1]{LLLS2019nonlinear}, one obtains that $q_1 = q_2$ in $M$, as desired.
\end{proof}

\vspace{10pt}

\noindent {\bf Acknowledgments.} 
T. L. and M. S. are supported by the Finnish Centre of Excellence in Inverse Modelling and Imaging, Academy of Finland grant 284715, and T.T. by grant 312119. M.S.\ was also supported by the Academy of Finland (grant 309963) and by the European Research Council under Horizon 2020 (ERC CoG 770924). Y.-H. L. is supported by the Ministry of Science and Technology Taiwan, under the Columbus Program: MOST-109-2636-M-009-006.

\bibliographystyle{alpha}
\bibliography{ref}

\begin{thebibliography}{LLLS20b}

\bibitem[AZ17]{AYT2017direct}
Yernat~M Assylbekov and Ting Zhou.
\newblock Direct and inverse problems for the nonlinear time-harmonic {M}axwell
  equations in {K}err-type media.
\newblock {\em arXiv preprint arXiv:1709.07767}, 2017.

\bibitem[Cal80]{calderon}
Alberto~P Calder{\'o}n.
\newblock On an inverse boundary value problem.
\newblock {\em Seminar in Numerical Analysis and its Applications to Continuum
  Physics (R\'{i}o de Janeiro: Soc. Brasileira de Matem\'{a}tica)}, pages
  65--73, 1980.

\bibitem[CF20]{carstea2020inverse}
C{\u{a}}t{\u{a}}lin~I C{\^a}rstea and Ali Feizmohammadi.
\newblock An inverse boundary value problem for certain anisotropic quasilinear
  elliptic equations, 2020.

\bibitem[CK20]{carstea2020recovery}
C{\u{a}}t{\u{a}}lin~I C{\^a}rstea and Manas Kar.
\newblock Recovery of coefficients for a weighted p-laplacian perturbed by a
  linear second order term, 2020.

\bibitem[CLL19]{CLL2017simultaneously}
Xinlin Cao, Yi-Hsuan Lin, and Hongyu Liu.
\newblock Simultaneously recovering potentials and embedded obstacles for
  anisotropic fractional {S}chr\"odinger operators.
\newblock {\em Inverse Problems and Imaging}, 13(1):197--210, 2019.

\bibitem[CNV19]{CNV2019reconstruction}
C{\u{a}}t{\u{a}}lin~I C{\^a}rstea, Gen Nakamura, and Manmohan Vashisth.
\newblock Reconstruction for the coefficients of a quasilinear elliptic partial
  differential equation.
\newblock {\em Applied Mathematics Letters}, 98:121--127, 2019.

\bibitem[FKSU09]{ferreira2009linearized}
David D.~S. Ferreira, Carlos Kenig, Johannes Sj\"ostrand, and Gunther Uhlmann.
\newblock On the linearized local {C}alder\'on problem.
\newblock {\em Math. Res. Lett.}, 16:955--970, 2009.

\bibitem[FO20]{FO20}
Ali Feizmohammadi and Lauri Oksanen.
\newblock An inverse problem for a semi-linear elliptic equation in
  {R}iemannian geometries.
\newblock {\em Journal of Differential Equations}, 296(6):4683–4719, 2020.

\bibitem[GST19]{guillarmousalotzou_complex}
Colin {Guillarmou}, Mikko {Salo}, and Leo {Tzou}.
\newblock {The linearized {C}alder\'on problem on complex manifolds}.
\newblock {\em Acta Mathematica Sinica, English Series}, 35(6):1043--1056,
  2019.

\bibitem[GT11]{guillarmou2011calderon}
Colin Guillarmou and Leo Tzou.
\newblock Calder{\'o}n inverse problem with partial data on {R}iemann surfaces.
\newblock {\em Duke Mathematical Journal}, 158(1):83--120, 2011.

\bibitem[Har06]{Hardy2006combinatorics}
Michael Hardy.
\newblock Combinatorics of partial derivatives.
\newblock {\em Electron. J. Comb}, 13(R1), 2006.

\bibitem[Hor85]{hormander1983analysis}
Lars Hormander.
\newblock {\em The Analysis of Linear Partial Differential Operators. {I-IV}}.
\newblock 1983-1985.

\bibitem[Isa90]{isakov1990inverse}
Victor Isakov.
\newblock {\em Inverse source problems}.
\newblock Number~34. American Mathematical Soc., 1990.

\bibitem[Isa93]{isakov1993uniqueness_parabolic}
Victor Isakov.
\newblock On uniqueness in inverse problems for semilinear parabolic equations.
\newblock {\em Archive for Rational Mechanics and Analysis}, 124(1):1--12,
  1993.

\bibitem[KKU20]{kian2020partial}
Yavar Kian, Katya Krupchyk, and Gunther Uhlmann.
\newblock Partial data inverse problems for quasilinear conductivity equations.
\newblock 2020.

\bibitem[KN02]{KN002}
Hyeonbae Kang and Gen Nakamura.
\newblock Identification of nonlinearity in a conductivity equation via the
  {D}irichlet-to-{N}eumann map.
\newblock {\em Inverse Problems}, 18:1079--1088, 2002.

\bibitem[KU20a]{krupchyk2020inverse}
Katya Krupchyk and Gunther Uhlmann.
\newblock Inverse problems for nonlinear magnetic schr\"odinger equations on
  conformally transversally anisotropic manifolds.
\newblock 2020.

\bibitem[KU20b]{KU2019partial}
Katya Krupchyk and Gunther Uhlmann.
\newblock Partial data inverse problems for semilinear elliptic equations with
  gradient nonlinearities.
\newblock {\em Mathematical Research Letters, to appear}, 2020.

\bibitem[KU20c]{KU2019remark}
Katya Krupchyk and Gunther Uhlmann.
\newblock A remark on partial data inverse problems for semilinear elliptic
  equations.
\newblock {\em Proc. Amer. Math. Soc.}, 148:681--685, 2020.

\bibitem[Lav67]{lavrentiev1967}
M.M. Lavrentiev.
\newblock {\em Some improperly posed problems of mathematical physics},
  volume~11.
\newblock Springer Tracts in Natural Philosophy, Springer, Berlin, 1967.

\bibitem[Lin20]{lin2020monotonicity}
Yi-Hsuan Lin.
\newblock Monotonicity-based inversion of fractional semilinear elliptic
  equations with power type nonlinearities.
\newblock {\em arXiv preprint arXiv:2005.07163}, 2020.

\bibitem[LL20]{LL2020inverse}
Ru-Yu Lai and Yi-Hsuan Lin.
\newblock Inverse problems for fractional semilinear elliptic equations.
\newblock {\em arXiv preprint arXiv:2004.00549}, 2020.

\bibitem[LLLS20a]{LLLS2019nonlinear}
Matti Lassas, Tony Liimatainen, Yi-Hsuan Lin, and Mikko Salo.
\newblock Inverse problems for elliptic equations with power type
  nonlinearities.
\newblock {\em Journal de Math{\'e}matiques Pures et Appliqu{\'e}es, in press},
  2020.

\bibitem[LLLS20b]{LLLS2019partial}
Matti Lassas, Tony Liimatainen, Yi-Hsuan Lin, and Mikko Salo.
\newblock Partial data inverse problems and simultaneous recovery of boundary
  and coefficients for semilinear elliptic equations.
\newblock {\em Revista Matematica Iberoamericana, accepted for publication},
  2020.

\bibitem[LLS19]{lassas2018poisson}
Matti Lassas, Tony Liimatainen, and Mikko Salo.
\newblock The {P}oisson embedding approach to the {C}alder\'on problem.
\newblock {\em Mathematische Annalen}, pages 1--49, 2019.

\bibitem[LO20]{LO2020fractional_lower}
Ru-Yu Lai and Laurel Ohm.
\newblock Inverse problems for the fractional {L}aplace equation with lower
  order nonlinear perturbations.
\newblock {\em arXiv preprint arXiv:2009.07883}, 2020.

\bibitem[LZ20]{LZ2020partial}
Ru-Yu Lai and Ting Zhou.
\newblock Partial data inverse problems for nonlinear magnetic {S}chr\"odinger
  equations.
\newblock {\em arXiv preprint arXiv:2007.02475}, 2020.

\bibitem[Rie38]{riesz1938}
M.~Riesz.
\newblock Integrales de riemann-liouville et potentiels.
\newblock {\em Acta Szeged}, 9:1--42, 1938.

\bibitem[RZ18]{RenZhang}
Kui Ren and Rongting Zhang.
\newblock Nonlinear quantitative photoacoustic tomography with two-photon
  absorption.
\newblock {\em SIAM J. Appl. Math.}, 78(1):479--503, 2018.

\bibitem[SU97]{sun1997inverse}
Ziqi Sun and Gunther Uhlmann.
\newblock Inverse problems in quasilinear anisotropic media.
\newblock {\em American journal of mathematics}, 119(4):771--797, 1997.

\bibitem[Sun96]{sun1996quasilinear}
Ziqi Sun.
\newblock On a quasilinear inverse boundary value problem.
\newblock {\em Mathematische Zeitschrift}, 221(1):293--305, 1996.

\bibitem[Tay11]{taylor2011partial}
Michael~E. Taylor.
\newblock {\em Partial differential equations {I}. {B}asic theory}, volume 115
  of {\em Applied Mathematical Sciences}.
\newblock Springer, New York, second edition, 2011.

\bibitem[Zei86]{Zeidler1986}
E.~Zeidler.
\newblock {\em Nonlinear functional analysis and its applications I:
  Fixed-point theorems}, volume~1.
\newblock Springer-Verlag, New York, 1986.

\end{thebibliography}

\end{document}